\documentclass{amsart}
\usepackage[scale=0.75, centering, headheight=14pt]{geometry}
\usepackage[latin1]{inputenc}
\usepackage[T1]{fontenc}
\usepackage{lmodern}
\usepackage[english]{babel}
\usepackage{microtype}
\usepackage{tikz}
\usepackage{amsmath,amssymb,amsfonts,amsthm}
\usepackage{mathtools,accents}
\usepackage{mathrsfs}
\usepackage{aliascnt}
\usepackage{braket}
\usepackage{bm}
\usepackage[initials]{amsrefs}

\usepackage[citecolor=blue,colorlinks]{hyperref}
\addto\extrasenglish{}

\usepackage{enumerate}
\usepackage{xcolor}

\usepackage{aliascnt}

\makeatletter
\def\newaliasedtheorem#1[#2]#3{
  \newaliascnt{#1@alt}{#2}
  \newtheorem{#1}[#1@alt]{#3}
  \expandafter\newcommand\csname #1@altname\endcsname{#3}
}
\makeatother

\theoremstyle{plain}
\newtheorem{theorem}{Theorem}[section]
\newaliasedtheorem{lemma}[theorem]{Lemma}
\newaliasedtheorem{prop}[theorem]{Proposition}
\newaliasedtheorem{claim}[theorem]{Claim}
\newaliasedtheorem{corollary}[theorem]{Corollary}

\theoremstyle{definition}
\newaliasedtheorem{definition}[theorem]{Definition}
\newaliasedtheorem{example}[theorem]{Example}

\theoremstyle{remark}
\newaliasedtheorem{remark}[theorem]{Remark}

\numberwithin{equation}{section}

\def\eps{\varepsilon}
\def\R{{\mathbb{R}}}
\def\T{{\mathbb T}}
\def\d{{\partial}}
\def\ra{\rightarrow}
\def\a{{\alpha}}
\def\i{{\infty}}

\title[Scattering for NLKG on $\R^d \times \mathbb T$]{Large Data Scattering for the  defocusing 
NLKG on waveguide $\R^d \times \mathbb T$.}

\author[L. Forcella \and L. Hari]{Luigi Forcella \and Lysianne Hari}

\address{Luigi Forcella\hfill\break Scuola Normale Superiore, Piazza dei Cavalieri, 7, 56126 Pisa, Italy}
\email{luigi.forcella@sns.it}

\address{Lysianne Hari \hfill\break Laboratoire de Math\'ematiques de Besan\c{c}on (LMB, UMR CNRS 6623), Universit\'e de Bourgogne Franche-Comt\'e, 16 route de Gray, 25030 Besan\c{c}on CEDEX, France.}
\email{lysianne.hari@univ-fcomte.fr}

\subjclass[2000]{}


\begin{document}

\begin{abstract}
We consider the pure-power defocusing nonlinear Klein-Gordon equation, in the $H^1-$subcritical case, posed on the product space 
$\R^d\times \mathbb T$, where $\mathbb T$ is the one-dimensional flat torus. In this framework, we prove that scattering holds for any initial data belonging to the energy space $H^1 \times L^2$ for $1\leq d\leq 4$. The strategy consists in proving a suitable profile decomposition theorem in $\R^d\times \mathbb T$ to pursue a concentration-compactness \& rigidity method.
\\[2mm]
\\[2mm]
\end{abstract}
\maketitle



\section{Introduction}

We consider the following Cauchy problem for the pure-power defocusing nonlinear Klein-Gordon equation posed on the waveguide $\mathbb{R}^d\times\mathbb T,$ with $1\leq d\leq 4$
\begin{equation}\label{NLKG}
\left\{\begin{aligned}
\partial_{tt}u-\Delta_{x,y}u+u&=-|u|^{\alpha}u,\quad (t,x,y) \in \R\times \R^d\times \mathbb T\\
u(0,x,y)&=f(x,y)\in H^1(\R^d\times \mathbb T)\\
\partial_tu(0,x,y)&=g(x,y)\in L^2(\R^d\times \mathbb T)
\end{aligned}\right.,
\end{equation}
where $\mathbb T$ is the one-dimensional flat torus and $\Delta_{x,y}=\Delta_x+\Delta_y$ is the usual  Laplace operator $\sum_{i=1}^d\partial_{x_i}^2+\partial_y^2$.

\noindent We consider nonlinearities that are energy subcritical on $\R^{d+1}$ and mass supercritical on $\R^d,$ 
namely we restrict our attention to $\frac4d<\alpha<\frac{4}{d-1}$ for $2\leq d\leq 4$ while $\a>4$ for $d=1.$ 
For some particular choices of nonlinearities, aside from the natural question of existence of solutions, it is of interest to try to relate the long-time behavior of nonlinear solutions to linear solutions in appropriate functional spaces. We wish to investigate the energy scattering for \eqref{NLKG}.
\\

About the pure euclidean case $\R^d$, there is a huge mathematical literature, not only for the Klein-Gordon equation but in general for other dispersive PDEs such as the nonlinear Schr\"odinger (NLS) equation and the nonlinear wave equation (NLW). 
We recall that Strichartz estimates play an essential role for the local well-posedness and for the large time analysis of the solutions - once Strichartz estimates have been proved to hold globally in time. 
The nonlinear Klein-Gordon (NLKG) equation has been deeply studied in the euclidean framework, producing a huge literature.
We only give here some references amongst others about the scattering results: in high dimension cases $d\geq 3$, we mention the early works by Morawetz \cite{Mor} and Morawetz and Strauss \cite{MS}, the works by Brenner \cite{Brenner84, Brenner85}, Ginibre and Velo \cite{GV_NLKG_I, GV_NLKG_II, GV_decay}, while for the low dimensional case $\mathbb R^d$ with $d=1,2,$ the question of scattering has been solved by Nakanishi in \cite{Nak}. The focusing case have been investigated in \cite{IMN11} by Ibrahim, Masmoudi and Nakanishi both in the energy subcritical and critical cases. 
For a more complete picture of the known results, we refer the reader to the references contained in the previously cited papers.  

Unlike the euclidean setting, the compact one does not exhibit the same phenomenon. This is due to the presence of periodic solutions inducing a lack of (global-in-time) summability on them. Nevertheless, it is worth pointing out that existence properties on compact manifolds have been investigated  for NLS by Bourgain in \cite{Bou} and later by Burq, Gerard and Tzvetkov in \cite{BGT}.
For existence results for NLKG, valid on more general manifolds, we refer the reader to \cite{KapIII}.

The question of ``mixing'' both configurations, to understand the competition of induced phenomena is natural. The study of scattering properties for solutions to NLS posed on a product space was proved for small data on $\R^d\times\mathcal{M}^k$ - $\mathcal{M}^k$ being a compact Riemannian manifold - by Tzvetkov-Visciglia \cite{TVI} followed by a theorem of large data scattering by the previous authors in \cite{TV2}. 
We also mention the existence of several related results in mixed settings, among which \cite{GPT,HP,HPTV,Tar,Rocha,CGYZ} and references therein.
\\


Our purpose is to carry on with the investigation of the second author and Visciglia started in \cite{HV}. In that paper the authors proved scattering for small energy data for the pure-power nonlinear energy-critical Klein-Gordon equation in the framework of $\R^d\times\mathcal M^2,$ in both defocusing and focusing regimes (the latter corresponding to an opposite sign in front of the nonlinear term in \eqref{NLKG}) and where $\mathcal M^2$ is a bidimensional compact manifold. For small initial data, once Strichartz estimates have been proved to hold globally in time, the global well-posedness and scattering can be proved by a perturbative argument.

This is no more the case when dealing with initial data without smallness assumption. We use the strategy of concentration-compactness \and rigidity method, pioneered by Kenig and Merle in \cite{KM1, KM2} for long-time behavior study for dispersive PDEs. To this aim, after having studied small data scattering on $\R^d \times \T$, our first step is to prove a profile decomposition theorem on the considered product space. Then thanks to a perturbative argument, we construct a minimal energy solution which is global in time but does not enjoy a finite Strichartz bound which would lead to the scattering property. 
Moreover, we prove that the trajectory of this solution is precompact in the energy space, and this will give a contradiction to its existence once combined with Nakanishi/Morawetz estimates. The choice of the strategy \emph{\`a la Kenig \& Merle} seems to be the best adapted to our setting, to deal with either defocusing or focusing nonlinearities and may be revisited for critical cases. We however recall that, combined with the mixed geometry, the ``bad sign'' of the energy in the focusing case is usually an obstruction to prove a priori bounds such as the Morawetz or the Nakanishi/Morawetz estimates, whereas for the energy-critical cases, it is expected that the lack of scaling invariance of NLKG will bring a delicate technical issue. Therefore, we restrict our attention on the defocusing (energy) subcritical cases, and especially on the adjustment of the euclidean arguments and tools in our setting, whereas the other cases are objects of future investigation.
\\

We briefly recall what we intend as scattering properties: we investigate the completeness of the wave operator by showing that, a global solution $u(t,x,t)$ to \eqref{NLKG} behaves, as time $t$ tends to $\pm \infty$, like a solution to the following linear equation
\begin{equation}\label{LKG}
\left\{\begin{aligned}
\partial_{tt}v-\Delta_{x,y}v+v&=0,\quad (t,x,y) \in \R\times \R^d\times \mathbb T\\
v(0,x,y)&=f^\pm\in H^1(\R^d\times \mathbb T)\\
\partial_tv(0,x,y)&=g^\pm\in L^2(\R^d\times \mathbb T)
\end{aligned}\right.
\end{equation}
for some initial data $(f^\pm,g^\pm)\in H^1(\R^d\times\mathbb T)\times L^2(\R^d\times\mathbb T).$ 
The main result of this paper is stated as follows. 
\begin{theorem}\label{main}
Assume that $d=1$ and $\alpha>4$ or $2\leq d\leq4$ and $\frac4d<\a<\frac{4}{d-1}.$ 
Let
\begin{equation}\label{reg-sol}
u\in \mathcal C(\R; H^1(\R^d\times\mathbb T))\cap \mathcal C^1(\R; L^2(\R^d\times\mathbb T))\cap L^{\alpha+1}(\R;L^{2(\alpha+1)}(\R^d\times\mathbb T)).
\end{equation}
be the unique global solution to \eqref{NLKG}: then for $t\to+\infty$ (respectively $t\to-\infty$) there exists $(f^+,g^+)\in H^1(\R^d\times\mathbb T)\times L^2(\R^d\times\mathbb T)$ (respectively $(f^-,g^-)\in H^1(\R^d\times\mathbb T)\times L^2(\R^d\times\mathbb T)$) such that 
\begin{equation}\label{scatt}
\lim_{t \rightarrow+\infty} \left\|u(t,x)-u^+(t,x) \right\|_{H^1(\R^d\times\mathbb T)} + \left\| \partial_t u(t,x)-\partial_t u^+(t,x) \right\|_{L^2(\R^d\times\mathbb T)}=0,
\end{equation}
\begin{equation*}
\left(\text{respectively}\quad\lim_{t \rightarrow - \infty} \left\|u(t,x)-u^-(t,x) \right\|_{H^1(\R^d\times\mathbb T)} + \left\| \partial_t u(t,x)-\partial_t u^-(t,x) \right\|_{L^2(\R^d\times\mathbb T)}=0\right),
\end{equation*}
where $u^+(t,x,y), u^-(t,x,y)\in H^1(\R^d\times\mathbb T)\times L^2(\R^d\times\mathbb T)$ are the corresponding solutions with initial data $(f^+,g^+)$ and $(f^-,g^-)$ to \eqref{LKG}.
\end{theorem}

\begin{remark}\label{Main-rmk}
The scattering property can be proved (in both small and large data cases) for $\alpha$ lying between the $L^2-$ critical exponent on $\R^d$ and the $H^1-$critical one on $\R^d \times \mathbb T$. In fact, considering data which are constant in their compact variable, it is easy to see that for $\alpha< \frac{4}{d}$, the analysis is reduced to the 
$L^2-$subcritical case on $\R^d$, for which no scattering in energy space is available. 

One may ask whether \eqref{scatt} is still true in critical cases and/or in the focusing case; due to technical considerations involving tools we use in the paper, the result does not contain the any consideration about the aforementioned cases.

We quickly sum-up the interval where $\alpha$ should lie to fulfill our assumptions:
\begin{center}
		\begin{tikzpicture}
	\draw (0,0) -- (10,0);
	\node at (1,0) {$\bullet$};
	\node at (2,0) {$\bullet$};
	\node at (8,0) {$\bullet$};
	\node at (4,0) {$\bullet$};
	\node[below] at (1,0) {$0$};
	\node[below] at (2,0) {$\frac{4}{d+1}$};
	\node[below] at (8,0) {$\frac{4}{d-1}$};
	\node[below] at (4,0) {$\frac{4}{d}$};
	\draw[dashed,<->](2,-1) -- (8,-1);
	\draw[dashed,<->](4,0.5) -- (8,0.5);
	\node[below] at (6,1) {\tiny Scattering on $\R^d\times \mathbb{T}$};
	\node[below] at (6,-0.5) {\tiny Scattering on $\R^{d+1}$ };
	\end{tikzpicture}
	\\[4mm]
	Picture for $\alpha$ such that \autoref{main} holds, with $1\leq d\leq 4$.
	\\[3mm]
\end{center}

Let us also notice that in the small data context, one can deal with both critical exponents $\frac{4}{d}$ and $\frac{4}{d-1}$, either in focusing and defocusing cases. Therefore, for small data, it is possible to add $d=5, \alpha =1,$ which is energy critical (see \autopageref{thm:small} for details).
\end{remark}

\subsection{Outline of the paper}
We present how the proof of \autoref{main} has been organized.\\
In the first section \autoref{sec:Strichartz}, we briefly sketch the proof of suitable global in time Strichartz estimates on the whole product space. We then deduce global existence of the solution to the Cauchy problem \eqref{NLKG} before concluding the section with small data scattering results, as it is the first step of the concentration-compactness scheme. \autoref{sec:prof-dec} presents a profile decomposition theorem, which in turn exhibits the existence of a non-trivial minimal energy soliton-like solution to \eqref{NLKG} in \autoref{sec:critical element}. It is a global non-scattering solution enjoying some compactness property. After the construction of this minimal element, we finally prove in \autoref{sec:rigidity} that by means of Nakanishi/Morawetz type estimates, this solution cannot exist. 

\subsection{Notations} Along the paper, space variable $x$ refers to the euclidean component of the product space $\R^d\times\mathbb T$, while $y$ belongs to the compact part: therefore $(x,y)\in\R^d\times\mathbb T.$ Consequently the notation $\Delta_{\R^d}$ and $\Delta_{\mathbb T}$ is used when we consider the restrictions of $\Delta$ on $\R^d$ and $\mathbb T,$ respectively. By analogous meaning, $\nabla$ stands for the $(d+1)$-components vector $\nabla=(\nabla_x,\partial_y).$\\

\noindent With $L^p:=L^p(\R^d\times\mathbb T)$ we mean the usual Lebesgue spaces and $L^{p}_x$ and $L^p_y$ for $L^p(\R^d)$ and $L^p(\mathbb T)$ respectively. The same holds for the Hilbert space $H^s:=H^s(\R^d\times\mathbb T)$ with compact notations $H^s_x:=H^s(\R^d)$ and $H^s_y:=H^s(\mathbb T).$\\

\noindent The Bochner space $L^p(I;X)$ is classically defined as the space of functions $f:I\subseteq\R\to X$ having finite $L^p(I;X)$ norm, where 
\begin{equation*}
\|f\|_{L^p(I;X)}:=\left(\int_I\|f\|^p_X(t)\,dt\right)^{1/p}.
\end{equation*}
If $I=\R$ we simply write $L^pX.$ For any real $p\geq1,$ we denote with $p^\prime$  its conjugate given by $\frac{1}{p^\prime}=1-\frac1p.$ 
For a vector $(f,g)$ we write $(f,g)^T=\begin{pmatrix} f \\ g\end{pmatrix}$ when convenient.\\
\noindent We indicate by $\mathcal F$ and $\mathcal F^{-1}$ the Fourier transform and its inverse, respectively, with respect to the $x-$variable. \\

\noindent The expressions $A\lesssim B$ or $A\gtrsim B$ mean that there exists a constant $C>0$ such that $A\leq CB$ or $A\geq CB,$ respectively, while $A\sim B$ means that both previous relations hold true. \\

\noindent In conclusion, we recall that the the energy is a conserved quantity for \eqref{NLKG}, namely
\begin{equation}
	\label{energy}
E(t)=E(u(t)) := \dfrac{1}{2}\left(\|\d_t u(t)\|_{L^2}^2+\|\nabla u(t)\|_{L^2}^2+\|u(t)\|_{L^2}^2+\dfrac{2}{\alpha+2}\| u(t)\|_{L^{\alpha+2}}^{\alpha+2}\right)= E(0),\quad \forall\, t\in\R.
\end{equation}

\section{Strichartz estimates}\label{sec:Strichartz}

In this section we prove some Strichartz estimates on the whole product space, and deduce global existence of the solution to the Cauchy problem \eqref{NLKG} in our setting, before handling small data scattering results. These results are the first step to perform a concentration-compactness method in the subsequent sections.

\noindent In our framework, energy conservation is not enough to handle the global existence problem and we need global in time Strichartz estimates, for which we do not have any restriction on the euclidean dimension $d.$ These are stated in the following theorem. 
\begin{theorem}[Strichartz estimates]\label{thm:strichartz}
Let $d\in\mathbb N$ and $1\leq q, r \leq \infty$ such that $(q,r)$ satisfies
\begin{equation}
\label{pairs}
\left\lbrace \begin{array}{lll}
\dfrac{2q}{q-4} \leq r, & q\geq 4, & \textrm{ if } d=1  \\ \\
\dfrac{2dq}{dq-4} \leq r \leq \dfrac{2q(d+1)}{q(d-1)-2}, &q> 2, & \textrm{ if } d=2 \quad q\geq 2,  \textrm{ if } d\geq 3 \\
\end{array} .\right.
\end{equation}
Let $w\in \mathcal C(\R; H^1)\cap \mathcal C^1(\R; L^2)$ be the unique solution to the following nonlinear problem:
\begin{equation}\label{w}
\left\{\begin{aligned}
\partial_{tt}u-\Delta u+u&=F,\quad (t,x,y) \in \R\times \R^d\times \mathbb T\\
u(0,x,y)&=f\in H^1\\
\partial_tu(0,x,y)&=g\in L^2
\end{aligned}\right.,
\end{equation}
\noindent where $F=F(t,x,y)\in L^1L^2.$ Then the estimate below holds:
\begin{equation}
\|w\|_{L^qL^r} \leq C \left(\|f\|_{H^1} + \|g\|_{L^2} 
+ \|F\|_{L^1L^2}\right).
\end{equation}
\end{theorem}
\begin{remark}\label{remark:strichartz}
Let us make some remarks about the conditions on the exponents.
The method we apply to obtain Strichartz estimates on the whole product space is used in \cite{HV} for the energy critical NLKG posed on $\R^d \times \mathcal{M}^2$. 
The method is divided into the following steps:
\begin{enumerate}
	\item we state the estimates on $\R^d$, involving Besov spaces;
	\item we use embedding theorems to deduce some estimates that hold in Lebesgue spaces posed on $\R^d;$
	\item we use a scaling argument to handle masses different from one;
	\item we write \eqref{w} in the basis of eigenfunctions of $\mathbb T$ and prove \autoref{thm:strichartz} in the fashion of \cite{TVI} 
	and \cite{HV}.
\end{enumerate}
In \cite{HV}, dealing with energy-critical nonlinearities, only critical embeddings were needed to prove small data scattering.
In our subcritical setting, one has to consider a wider range of Strichartz estimates to prove such results, obtained with ``subcritical'' embeddings. Deeper discussions about these estimates will be made along the proof of \autoref{thm:strichartz}.
\end{remark}

Global existence of the solution, namely  point $(1)$ of \autoref{main} then follows easily: a standard contraction principle 
performed on a small time interval $T=T(\|f\|_{H^1}+\|g\|_{L^2})$ implies local well-posedness on suitable Banach spaces. Energy conservation \eqref{energy} gives therefore global existence by time-stepping, since we are in the defocusing case.
\\
We do not write the details of the proof of global existence in this paper since it does not require any tricky computation.
\\

\noindent We are now able to deal with small data scattering problem.
\begin{theorem}[Small data scattering]
\label{thm:small}
Let $d=1$ and $\alpha\geq 4$ or $2\leq d\leq 5$ and $\alpha$ be such that $\frac{4}{d}\leq\alpha\leq\frac{4}{d-1}$.
Then there exists $\eps >0$ such that for all $(f,g)\in H^1\times L^2$ satisfying $\|f\|_{H^1}+\|g\|_{L^2}< \eps$,
the global nonlinear solution $u$ to the Cauchy problem \eqref{NLKG}
\begin{equation*}
u\in \mathcal C(\R; H^1)\cap \mathcal C^1(\R; L^2)\cap L^{\alpha+1}L^{2(\alpha+1)}
\end{equation*}
scatters in the sense of \eqref{scatt}.
\end{theorem}
\begin{remark}
It is worth mentioning that the analysis for small initial data can be stated without any further restriction in the focusing case, namely replacing in \eqref{NLKG} the sign in front of the nonlinear term with a plus sign. Furthermore, observe that the result of the theorem above  is valid also in the critical cases. The main restriction on $\alpha$ is carried by the fact that $(\alpha+1, 2\alpha+2)$ should satisfy \eqref{pairs}. It is easy to check that $d=5, \alpha=1$ is the only case that can be handled for $d>4$ and it is critical.
\end{remark}
In order to prove \autoref{thm:strichartz} we recall the definition of the Besov spaces. Given a cut-off function $\chi_0$ such that 
\begin{equation*}
C^{\infty}_c(\R^d;\R)\ni\chi_0(\xi)=
\begin{cases}
1 & \text{if}\quad |\xi|\leq1\\
0 & \text{if}\quad |\xi|>2
\end{cases}
\end{equation*} 
\noindent then  are defined the following dyadic functions
\begin{equation*}
\varphi_j(\xi)=\chi_0(2^{-j}\xi)-\chi_0(2^{-j+1}\xi),
\end{equation*}
\noindent yielding to the partition of the unity $$\chi_0(\xi)+\sum_{j>0} \varphi_j(\xi)=1 \quad \forall \xi\in\R^d.$$

\noindent By denoting with $\mathcal S(\R^d)$ the set of all tempered distributions on $\R^d,$ and with $P_j,\, j\in\mathbb N\cup\{0\},$ are defined as the following operators:
\begin{align*}
 P_0 f & := \mathcal{F}^{-1}\left(\chi_0 \mathcal{F}(f)\right), \\
 P_j f & := \mathcal{F}^{-1}\left(\varphi_j \mathcal{F}(f)\right), \quad \forall j\in\mathbb N.
\end{align*}

\noindent Let $-\infty < s < \infty.$ Then, for $0<q\leq \infty $, the Besov space $B^s_{q,2}$ is defined by
  $$B^s_{q,2}(\R^d) = \left\lbrace f\in \mathcal S(\R^d) \Big|  
  \left\{2^{js}\|P_j f\|_{L^q(\R^d)}\right\}_{j\in\mathbb N\cup\{0\}}\in l^2  
  \right\rbrace,$$
\noindent where $l^2$ is the classical space of square-summable sequences. \\


\noindent We rigorously prove the steps listed in \autoref{remark:strichartz} in order to prove \autoref{thm:strichartz}.\\

\begin{proof}[Proof of \autoref{thm:strichartz}] We introduce the number $s\in [0,1]$ given by
\begin{equation}
\label{s}
s = 1- \dfrac{1}{2} \left(\dfrac{d}{2}+1\right)\left(\dfrac{1}{r'}-\dfrac{1}{r}\right)= 
1- \dfrac{1}{2} \left(\dfrac{d}{2}+1\right)\left(1-\dfrac{2}{r}\right). 
\end{equation}

\noindent \textbf{Step 1.} We begin with the following proposition which is given in a pure euclidean context. 
\begin{prop}[Strichartz estimates for the euclidean case (from \cite{NS11})]
\label{strichartzRd}
Let $d\geq 1,$ $s$ as in \eqref{s} and $q,\rho\geq 2$ such that 
	\begin{equation}\label{adm}
\begin{array}{ll}
    \dfrac{2}{q}=d\left(\dfrac{1}{2}-\dfrac{1}{\rho}\right) \quad&\textrm{ if }\quad d\geq3, \\
	q>2\quad&\textrm{ if }\quad d=2,  \\
	q\geq 4 \quad&\textrm{ if }\quad d=1.
	\end{array}
	\end{equation}
 
\noindent Consider $w=w(t,x)$ satisfying 
\begin{equation}\label{www}
\left\{\begin{aligned}
\partial_{tt}w-\Delta_{\R^d}w+w&=F,\quad (t,x) \in \R\times \R^d\\
w(0,x)&=f\in H^1(\R^d)\\
\partial_tw(0,x)&=g\in L^2(\R^d)
\end{aligned}\right.,
\end{equation}
where $F=F(t,x)\in L^1(\R;L^2(\R^d)).$ Then
\begin{equation}
\label{wRd}
\left\|w\right\|_{L^q(\R;{B^{s}_{\rho,2}(\R^d))}} \leq C \left( \|f\|_{H^1(\R^d)} + \|g\|_{L^2(\R^d)} 
+ \|F\|_{L^1(\R;L^2(\R^d))} \right),
\end{equation}
where $C>0$ depends only on the choice of the pair $(q,r)$ and on the dimension $d.$
\end{prop}
The proof is detailed in \cite{NS11}, using previous results from \cite{Brenner84, Brenner85, GV_NLKG_I, GV_decay, GV_NLKG_II, IMN11, KT, NOpisa, NS11, Pecher85} 
 (see for example \cite{GV_decay, GV_NLKG_I, GV_NLKG_II, KT, NS11} for the proofs) and \cite{KT} (see also \cite{MNO02,MNO03}) for the endpoint cases when $d\geq 3$.\\
 The estimates in previous works are more general: the source term can be handled in a ``dual'' Besov space. 
 We chose to handle the source term in the
  only homogeneous space we can work with, using a scaling method (see \cite{HV} for deeper discussions).
 \\
 
\noindent \textbf{Step 2.} We state the following embedding theorem contained in \cite{Triebel_II, Triebel_III} and references therein.
\begin{theorem}[Embedding theorems]
\label{embedding}
Let $d\geq 1$, $s>0$, and $2\leq r,\rho\leq \infty.$ Consider the Besov space $B^s_{\rho,2}(\R^d)$ and the Lebesgue space $L^r(\R^d)$. Then  the embedding relations below hold:
\begin{enumerate}
\item	$B^s_{\rho,2}(\R^d) \hookrightarrow L^{\rho}(\R^d);$
\item   If $\rho^*:=\dfrac{d\rho}{d-s\rho}$, when $d>s\rho^*,$ then 
$B^s_{\rho,2}(\R^d) \hookrightarrow L^{r}(\R^d)$ for $ \rho\leq r \leq \rho^*$;
\item If $d\leq s\rho$, then $B^s_{\rho,2}(\R^d) \hookrightarrow L^{r}(\R^d)$ for $ \rho\leq r <+\infty.$
\end{enumerate}




\end{theorem}

\noindent Performing quick computations, we notice that for $s$ satisfying \eqref{s} and $(q,\rho)$ as in \eqref{adm}, the conditions of the theorem yield 
	\begin{equation*}
\begin{array}{ll}
    d-s\rho <0 \quad&\textrm{ if }\quad d= 1, \\
	d-s\rho=0\quad&\textrm{ if }\quad d= 2,  \\
	d-s\rho>0 \quad&\textrm{ if }\quad d\geq 3.
	\end{array}
	\end{equation*}
By \cite{Triebel}, we have in our setting $B^s_{\rho,2}(\R^d)\subset W^{s,\rho}(\R^d)$. Thus with Sobolev embeddings	
$W^{s,\rho}(\R^d) \subset L^{r}(\R^d)$ with 
\begin{equation*}
\begin{array}{ll}
r \in [\rho, \infty] \quad& \textrm{ if }\quad d=1 \\
r \in [\rho, \infty) \quad& \textrm{ if }\quad d=2 \\
r \in [\rho, \rho^*] \quad& \textrm{ if }\quad d\geq 3	
\end{array},
\end{equation*}
and computing $\rho, \rho^*$ in terms of $q$, we obtain
\begin{align}\label{condition1}
\dfrac{2dq}{dq-4}\leq r \qquad \qquad\qquad &\textrm{ if }d= 1,2, \\\label{condition2}
\dfrac{2dq}{dq-4}\leq r \leq \dfrac{2d^2q}{d^2q-2d-2dq+4} \quad &\textrm{ if }d\geq 3.
\end{align}	
Strichartz estimates involving Lebesgue spaces instead of Besov spaces follow immediately by applying the previous embedding theorem to \autoref{strichartzRd}, with $r$ satisfying \eqref{condition1} or \eqref{condition2}. 
\\
	
\noindent \textbf{Step 3.} By defining  $w_\lambda:=w\left(\sqrt{\lambda}t, \sqrt{\lambda}x\right) $ with $w$ as in \eqref{www} and noticing that it satisfies 
\begin{equation}\label{wlambda}
\left\{\begin{aligned}
\partial_{tt}w_\lambda-\Delta_{\R^d}w_\lambda+\lambda w_\lambda&=F_\lambda,\quad (t,x) \in \R\times \R^d\\
w_\lambda(0,\cdot)&=f_\lambda\\
\partial_tw_\lambda(0,\cdot)&=g_\lambda
\end{aligned}\right.,
\end{equation}
\noindent where \begin{equation*}
        f_\lambda(x) =  f\left(\sqrt{\lambda}x\right), \quad 
       g_\lambda(x) = \sqrt{\lambda}g\left(\sqrt{\lambda}x\right), \quad
       F_\lambda(t,x) = \lambda F\left(\sqrt{\lambda}t, \sqrt{\lambda}x\right),
       \end{equation*}
we can easily prove the next result, whose detailed proof can be found in \cite{HV}.
 \begin{prop}
Consider a pair $(q,\rho)$ as in \eqref{adm}, $s$ given by \eqref{s} and $r$ as in \autoref{embedding}.
 Consider $w$ given by \eqref{www} for which \autoref{strichartzRd} holds. Then one has for \eqref{wlambda}
   \begin{equation}
  \label{Slambda}
\lambda^{\frac{1}{2}\left(\frac{d}{r}+\frac{1}{q}-\frac{d}{2}+1\right)}\|w_\lambda\|_{L^{q}(\R;L^{r}(\R^d))} 
\leq C \left(\sqrt{\lambda}\|f_\lambda\|_{L^2(\R^d)}+  \|f_\lambda\|_{\dot{H}^1(\R^d)} + \|g_\lambda\|_{L^2(\R^d)} + 
 \|F_\lambda\|_{L^1(\R;L^2(\R^d))}\right).  
 \end{equation}
 \end{prop}

\noindent \textbf{Step 4.} Once we can rely on the ingredients of the previous steps, we finally use the strategy from \cite{TVI} to conclude with the desired result. We write $\left\lbrace \lambda_j  \right\rbrace_{j\geq0}$ for
the eigenvalues of $-\Delta_{\mathbb T}$, sorted in ascending order and taking in account their multiplicities; we also introduce $\left\lbrace \Phi_j(y) \right\rbrace_{j\geq0}$, the eigenfunctions associated with $\lambda_j$, i.e.
\begin{equation}
         \label{basis}
         -\Delta_{\mathbb T} \Phi_j = \lambda_j \Phi_j, \quad \lambda_j\geq 0,\quad j\in \mathbb N \cup \{0\}.
        \end{equation}
This provides an orthonormal basis of $L^2\left(\mathbb T\right)$.
We now consider the solution to \eqref{w}
 and we write the functions in terms of \eqref{basis}:
\begin{equation}\label{decompo}
\begin{aligned}
 w(t,x,y)  &= \sum_{j=0}^\infty w_j(t,x) \Phi_j(y), \\ F(t,x,y)  &= \sum_{j=0}^\infty F_j(t,x) \Phi_j(y), \\
  f(x,y) &= \sum_{j=0}^\infty f_j(x) \Phi_j(y), \\g(x,y)  &= \sum_{j=0}^\infty g_j(x) \Phi_j(y), 
\end{aligned}
\end{equation} 
with $w_j=w_j(t,x)$ satisfying
\begin{equation}
 \label{flatNLKG}
 \partial_{tt} w_j - \Delta_{\R^d} w_j + w_j + \lambda_j w_j = F_j, \quad w_j(0,\cdot)= f_j, \quad 
 \d_t w_j(0,\cdot)= g_j.
\end{equation}
Taking $\lambda = 1+\lambda_j$ in \eqref{Slambda} it follows that
\begin{multline*}
    \left(\lambda_j+1\right)^{\frac{1}{2}\left(\frac{d}{r}+\frac{1}{q}+1-\frac{d}{2}\right)} \|w_j\|_{L^{q}(\R;L^{r}(\R^d))}  \leq  C\left( 
    (\lambda_j+1)^{1/2} \|f_j\|_{L^2(\R^d)} + \|f_j\|_{\dot{H}^1(\R^d)}+\|g_j\|_{L^2(\R^d)} \right.\\ \left.+ \|F_j\|_{L^{1}(\R;L^{2}(\R^d))}\right).
 \end{multline*}
Then, summing in $j$ the squares as in \cite{TVI} and \cite{HV} one obtains 
\begin{multline*}
  \left\|\left(\lambda_j+1\right)^{\frac{1}{2}\left(\frac{d}{r}+\frac{1}{q}+1-\frac{d}{2}\right)}w_j\right\|_{l^2_jL^{q}(\R;L^{r}(\R^d))}  
  \leq C \left( \left\|\left(\lambda_j+1\right)^{1/2} f_j\right\|_{l^2_jL^2(\R^d)} +
  \|f_j\|_{l^2_j\dot{H}^1(\R^d)}+\|g_j\|_{l^2_jL^2(\R^d)} \right. \\ \left.+ \|F_j\|_{l^2_jL^{1}(\R;L^2(\R^d))}\right). 
\end{multline*}
Since $\max(1,2)\leq 2 \leq \min (q,\rho),$ Minkowski inequality can be applied
 \begin{equation*} 
 \left\|\left(\lambda_j+1\right)^{\frac{1}{2}\left(\frac{d}{r}+\frac{1}{q}+1-\frac{d}{2}\right)}w_j\right\|_{L^{q}(\R;L^{r}(\R^d))l^2_j}  
 \leq C \left( \left\|\left(\lambda_j+1\right)^{1/2} f_j\right\|_{L^2(\R^d)l^2_j} + \|g_j\|_{L^2(\R^d)l^2_j}
 +\left\|F_j \right\|_{L^{1}(\R;L^{2}(\R^d))l^2_j}\right), 
  \end{equation*}
  
and by the Plancherel identity 
 +  
 one is able to handle the $y-$variable to obtain
  \begin{align*}
   \left\|(1-\Delta_y)^{\frac{1}{2}\left(\frac{d}{r}+\frac{1}{q}+1-\frac{d}{2}\right)}w\right\|_{L^q(\R;L^r(\R^d)L^2(\mathbb T))} & \leq C \left( 
  \|f\|_{{H}^1(\R^d\times \mathbb T)} + \|g\|_{L^2(\R^d\times \mathbb T)}+\|F\|_{L^1(\R;L^2(\R^d\times \mathbb T))}\right)\\  
 \end{align*}
 \noindent which in turn implies 
 \begin{align*}  \left\|w\right\|_{L^q(\R;L^r_x(\R^d) H^{\gamma}_y(\mathbb T))} & 
   \leq C \left(\|f\|_{{H}^1(\R^d\times \mathbb T)} + \|g\|_{L^2(\R^d\times \mathbb T)}+\|F\|_{L^1(\R;L^2(\R^d)L^2(\mathbb T))}\right),
  \end{align*}
 where \begin{equation*}
        \gamma = \left(\frac{d}{r}+\frac{1}{q}+1-\frac{d}{2}\right).
       \end{equation*}
 We easily see that $\gamma \geq 0$ when $d=1,2$, and by computing the condition $\gamma\geq 0$ for $d\geq 3$, we have
 \begin{align*}
 	\gamma\geq 0 &\iff \dfrac{2dq+2r+2qr-dqr}{2qr}\geq 0 \\
 	&\iff 2dq+2r+2qr-dqr \geq 0 \\
 	& \iff r \leq \dfrac{2dq}{dq-2q-2}.
 \end{align*}
It is easy to check that $\dfrac{2dq}{dq-2q-2} \geq \rho^*$ - the latter being larger than $r$ - which established that under \eqref{condition1},\eqref{condition2}, $\gamma$ is always nonnegative.
The proof is then completed by using a Sobolev embedding available for $\gamma\geq0$ 
\begin{equation}
\label{SobolevCompact}
H^\gamma(\mathbb T) \hookrightarrow L^r(\mathbb T)
\end{equation}
 which holds (at least) under one of the following conditions:
\begin{equation}\label{condition3} 
\begin{array}{ll}
\bullet \quad 2\gamma<1,  \dfrac{2}{1-2\gamma} \geq r,
	& \textrm{which is the ``usual'' condition to have Sobolev embedding},\\ \\
	\bullet \quad 2\gamma \geq 1, \quad r \geq 2, 
	& \textrm{which ensures}\, \eqref{SobolevCompact}\,\textrm{ with\, } H^\gamma(\mathbb T)\hookrightarrow L^\infty(\mathbb T)\,
	\textrm{\,allowing to control any} \\ & \textrm{ $L^r$ norm with the $H^\gamma$ norm since $\mathbb T$ is of finite volume.}
\end{array}
\end{equation}
Then by gluing together all conditions \eqref{condition1},\eqref{condition2},\eqref{condition3} in terms of $q$, we exhibit the exponent $r$ for which the Strichartz estimates can be proved:
for $d=1$, since $\gamma >1/2$, we have $H^\gamma(\mathbb T) \hookrightarrow L^\infty(\mathbb T)$ and so
\begin{equation*}
	\dfrac{2q}{q-4}\leq r.
\end{equation*}
For $d\geq 2$
\begin{equation*}
	\dfrac{2dq}{dq-4}\leq r \leq \min\left\{\dfrac{2d^2q}{d^2q-2d-2dq+4}, \dfrac{2q(d+1)}{dq-q-2}\right\},
\end{equation*}
that is
\begin{equation*}
\dfrac{2dq}{dq-4}\leq r \leq \dfrac{2q(d+1)}{dq-q-2},
\end{equation*}
which concludes the proof of \autoref{thm:strichartz}.
\end{proof}
As already introduced before, the tool given by Strichartz estimates implies the small data scattering, which holds also in the critical cases. 
\begin{proof}[Proof of \autoref{thm:small}]
We recall that in the framework of \autoref{thm:small}, we consider $\frac{4}{d}\leq\alpha\leq \frac{4}{d-1}$ for $2\leq d\leq 5$ and $\a\geq4$ if $d=1$. We handle both focusing and defocusing nonlinearities arguing as in \cite{HV}.
\\ 
We rewrite \eqref{NLKG} in the vector form. More precisely if $u$ is a solution to \eqref{NLKG} then the vector $(u,\d_tu)^T$ satisfies 
\begin{equation*}\d_t\begin{pmatrix}
         u \\ \d_t u
        \end{pmatrix}
 = \begin{pmatrix}
    0 & 1 \\ -\Delta+1 & 0
   \end{pmatrix}\begin{pmatrix}
         u \\ \d_t u
        \end{pmatrix} + \begin{pmatrix}
         0 \\ \pm|u|^{\alpha}u
        \end{pmatrix}.
\end{equation*}
We have that the following exponential matrix operator  
\begin{equation}\label{exp-matrix} 
e^{tH}= \begin{pmatrix}
    \cos\left(t\cdot \sqrt{1-\Delta}\right) & \dfrac{\sin \left(t\cdot\sqrt{1-\Delta}\right)}{\sqrt{1-\Delta}} 
    \\ \\
    -\sin \left(t\cdot\sqrt{1-\Delta}\right)\cdot\left(\sqrt{1-\Delta}\right) &  \cos\left(t\cdot\sqrt{1-\Delta}\right)
   \end{pmatrix},
    \end{equation} 
    
\noindent    is unitary on the energy space $H^1\times L^2$ (see \cite{NS11}). Moreover   
\begin{equation*} 
\begin{aligned}\begin{pmatrix}
         u \\ \d_t u
        \end{pmatrix}
  = e^{tH}\begin{pmatrix}
         f \\ g
        \end{pmatrix}+       
   \int_0^te^{(t-s)H}\begin{pmatrix}
         0 \\ |u|^{\alpha}u
        \end{pmatrix}ds 
        \end{aligned}
\end{equation*}
and then, since $e^{tH}$ is skew self-adjoint
 \begin{equation*} 
\begin{aligned}
e^{-tH}\begin{pmatrix}
         u \\ \d_t u
        \end{pmatrix}=
  \begin{pmatrix}
         f \\ g
        \end{pmatrix}+       
   \int_0^te^{-sH}\begin{pmatrix}
         0 \\ |u|^{\alpha}u
        \end{pmatrix}ds.
\end{aligned}
\end{equation*}
We now write $ \vec V(t)= e^{-tH}\begin{pmatrix}
         u \\ \d_t u
        \end{pmatrix}$, and consider $0<\tau<t$. Then 
   \begin{equation*}
   \|\vec V(t)-\vec V(\tau)\|_{H^1\times L^2}\leq C \int_{\tau}^{t} \||u|^{\alpha}u (s)\|_{L^2} ds \leq C \|u\|^{\a+1}_{L^{\alpha+1}([\tau,t],L^{2(\alpha+1)})},
   \end{equation*}
    and it is obvious that $\|u\|^{\a+1}_{L^{\alpha+1}([\tau,t],L^{2(\alpha+1)})}$ tends to zero as $t,\tau$ tends to infinity, since the solution belongs to 
    $L^{\alpha+1}\left(\R;L^{2(\alpha+1)}\right)$.
    
     Therefore, there exist 
   $(f^\pm,g^\pm)\in H^1\times L^2$ such that $\vec V(t)\rightarrow \begin{pmatrix}
         f^\pm \\ g^\pm
        \end{pmatrix}$ in $H^1\times L^2$ as $t\rightarrow \pm \infty$.
\\[5mm]
\end{proof}
\section{Profile Decomposition Theorem}\label{sec:prof-dec}
In this section we follow the arguments of \cite{BV, NS11} to provide a profile decomposition theorem which is the main ingredient in the proof of scattering properties in the whole energy space. 
\\
We start with the following preliminary lemma.  We use the following convention  
\begin{equation}\label{sob}
2^*=\begin{cases} \frac{2(d+1)}{d-1},&\quad\text{if}\quad d\geq2\\
+\infty,&\quad\text{if}\quad d=1
\end{cases}.
\end{equation} 
\begin{lemma}\label{preli}
 Let $\{v_n(x,y)\}_{n\in\mathbb N}\subset H^{1}(\R^d\times\mathbb T),$ with $1\leq d\leq4,$ be a bounded sequence. Define the set 
\begin{equation}\label{def}
\begin{aligned}
\Lambda(v_n)=\{w(x,y)\in L^2\,\big|\, & \exists\, \{(x_n,y_n)\}_{n\in\mathbb N}\subset\mathbb{R}^d\times \mathbb T\quad\textit{such that,}\\
&\textit{up to subsequence,}\quad v_{n_k}(x-x_{n},y-y_n)\overset{L^2_{x,y}}\rightharpoonup w(x,y)\}
\end{aligned}
\end{equation}
and let 
\begin{equation}\label{lambda}
\lambda(v_n)=\sup_{w\in\Lambda(v_n)}\|w\|_{L^2_{x,y}}.
\end{equation}
Then, for any $q$ such that $2q\in(2,2^*)$ we have
\begin{equation*}
\limsup_{n\ra\infty}\|v_n\|_{L^{2q}_{x,y}}\lesssim \lambda^{e}(v_n)
\end{equation*}
\end{lemma}
\noindent where
\begin{equation*}
e=e(q,d)=\frac{q-1}{3-5q}\left(\frac{q(d-1)-(d+1)}{q}\right)>0.
\end{equation*}
\begin{proof}

By the Sobolev embedding theorem, see \cite{Heb1}, the energy space embeds continuously in the Lebesgue space $L^{2^*}$. In particular $H^1(\R^d\times\mathbb T)\hookrightarrow L^{2q}(\R^d\times\mathbb T)$ for any $q\in[1,2^*/2]$ if $d\geq2$ while $q\geq1$ if $d=1,$ where $2^*$ is defined in \eqref{sob}. Consider, as Fourier multiplier, a cut-off function in the flat-frequencies space $\R^d_\xi$ where the cut-off is given by 
\begin{equation*}
C^{\infty}_c(\R^d;\R)\ni\chi(\xi)=
\begin{cases}
1 & \text{if}\quad |\xi|\leq1\\
0 & \text{if}\quad |\xi|>2
\end{cases}.
\end{equation*} 
\noindent By setting $\chi_R(\xi)=\chi(\xi/R),$ $R>0,$ we define the pseudo-differential operator with symbol $\chi_R.$ It is given by $\chi_R(|D|)f=\mathcal F^{-1}(\chi_R\mathcal Ff)(x)$ and similarly we define the operator  $\tilde{\chi}_R(|D|)$ with the associated symbol given by $\tilde{\chi}_R(\xi)=1-\chi_R(\xi).$ Later on we will also use the well-known properties 
\begin{equation}\label{fourier-prop}
\begin{aligned}
\mathcal F(fg)&=\mathcal F(f)\star\mathcal F(g)\\ 
\mathcal F(f(\sigma\cdot))&=\sigma^{-d}\mathcal Ff(\cdot/\sigma)
\end{aligned}
\end{equation}
which hold for any smooth functions $f,g:\R^d\to\R$ and any nonnegative real number $\sigma.$
In order to applye the Hausdorff-Young inequality $\mathcal F:L^p\to L^{p^\prime}$ for any $p\in[1,2],$ we set $2q=p^\prime$ and $p=\frac{p^\prime}{p^\prime-1}=\frac{2q}{2q-1}\in(1,2)$. We then use H\"older inequality with $\frac 1p=\frac 12+\frac1r$ and exploiting the precise structure of $\mathbb T$ we can write, for every $n\in\mathbb N,$
\begin{equation}\label{deco}
v_n(x,y)=\sum_{k\in\mathbb Z}v^k_n(x)e^{iky},
\end{equation}
where the functions $v^k_n$ are the Fourier coefficients, and similarly  
\begin{equation*}
\tilde{\chi}_R(|D|)v_n(x,y)=\sum_{k\in\mathbb Z}\tilde{\chi}_R(|D|)v^k_n(x)e^{iky}.
\end{equation*}
We first notice the embedding $ H^{\frac{1}{2}-\frac{1}{2q}}(\T) \hookrightarrow L^{2q}(\T)$ allowing us to write
\begin{equation}\label{not-localiz}
\begin{aligned}
\|\tilde{\chi}_R(|D|)v_n\|^2_{L^{2q}_{x,y}}&\lesssim\|\tilde{\chi}_R(|D|)v_n\|^2_{L^{2q}_xH_y^{\frac{1}{2}-\frac{1}{2q}}}=
\left\|\sum_{k\in\mathbb Z}\langle k\rangle^{1-\frac{1}{q}}|\tilde{\chi}_R(|D|)v^k_n|^2\right\|_{L^q_x},\\
&\lesssim\sum_{k\in\mathbb Z}\langle k\rangle^{1-\frac{1}{q}}\|\tilde{\chi}_R(|D|)v^k_n\|^2_{L^{2q}_x}\lesssim\sum_{k\in\mathbb Z}\langle k\rangle^{1-\frac{1}{q}}\|\mathcal{F}^{-1}(\tilde{\chi}_R(|\xi|)\hat{v}^k_n)(x)\|^2_{L^{2q}_x},\\
&\lesssim\sum_{k\in\mathbb Z}\langle k\rangle^{1-\frac{1}{q}}\|\tilde{\chi}_R(|\xi|)\hat{v}^k_n(\xi)\|^2_{L^{2q/(2q-1)}_\xi},\\
&\lesssim\sum_{k\in\mathbb Z}\langle k\rangle^{1-\frac{1}{q}}\|\langle \xi\rangle^{\frac{1}{2}+\frac{1}{2q}}\hat{v}^k_n(\xi)\|^2_{L^2_\xi}
\|\tilde{\chi}_R(|\xi|)\langle\xi\rangle^{-\frac{1}{2}-\frac{1}{2q}}\|^2_{L^{2q/(q-1)}_\xi},\\
\end{aligned}
\end{equation}
where an H\"older inequality was used in the last step. 
We notice that the last factor in the r.h.s. term is easily controlled as follows:
\begin{equation*}
\begin{aligned}
\|\tilde{\chi}_R(|\xi|)\langle\xi\rangle^{-\frac{1}{2}-\frac{1}{2q}}\|^2_{L^{2q/(q-1)}_\xi}&\lesssim\left(\int_{|\xi|\geq R}\frac{d\xi}{(1+|\xi|^2)^{\left(\frac14+\frac {1}{4q}\right)\left(\frac{2q}{q-1}\right)}}\right)^{\frac{q-1}{q}},\\
&\lesssim\left(\int_{|\xi|\geq R}\frac{d\xi}{|\xi|^{\frac{q+1}{q-1}}}\right)^{\frac{q-1}{q}},\\
&\lesssim\left(\int_R^{\infty}\frac{1}{\rho^{\frac{q+1}{q-1}-d+1}}\,d\rho\right)^{\frac{q-1}{q}},\\
&\lesssim \left(R^{d-\frac{q+1}{q-1}}\right)^{\frac{q-1}{q}}= R^{\frac{d(q-1)}{q}-\frac{q+1}{q}},
\end{aligned}
\end{equation*}
where the integrability of the term has been checked and $\frac{d(q-1)}{q}-(1+\frac{1}{q})<0.$ 
Thus, by the Plancherel identity, it ay be concluded that the estimate \eqref{not-localiz} satisfies:
\begin{equation*}
\begin{aligned}
\|\tilde{\chi}_R(|D|)v_n\|^2_{L^{2q}_{x,y}}&\lesssim R^{\frac{d(q-1)}{q}-(1+\frac{1}{q})}\sum_{k\in\mathbb Z}\langle k\rangle^{1-\frac{1}{q}}\int\langle\xi\rangle^{1+\frac{1}{q}}|\hat{v}^k_n(\xi)|^2\,d\xi,\\
&\lesssim R^{\frac{d(q-1)}{q}-(1+\frac{1}{q})}\|v_n\|_{H^1_{x,y}}^2.
\end{aligned}
\end{equation*} 
Recalling that $	\{v_n\}_{n\in\mathbb N}$ is bounded in $H^1,$ we summarize with 
\begin{equation*}
\begin{aligned}
\|\tilde{\chi}_R(|D|)v_n\|_{L^{2q}_{x,y}}&\lesssim R^{\frac{d(q-1)}{2q}-\frac{q+1}{2q}}=R^{\frac{q(d-1)-(d+1)}{2q}}.
\end{aligned}
\end{equation*} 

\noindent We now use \eqref{deco} and we define the localized part of $v_n$ as
\begin{equation*}
\chi_R(|D|)v_n(x,y)=\sum_{|k|\leq M}\chi_R(|D|)v_n^k(x)e^{iky}+\sum_{|k|>M}\chi_R(|D|)v_n^k(x)e^{iky}:=\chi^{\leq M}_R(|D|)v_n+\chi^{>M}_R(|D|)v_n.
\end{equation*}
We estimate the tail $\chi^{>M}_R(|D|)v_n$ as follows. By means of Minkowski and Cauchy-Schwartz inequalities we get

\begin{equation*}
\begin{aligned}
\|\chi^{>M}_R(|D|)v_n\|_{L^2_{x,y}}&\leq C(Vol(\mathbb T))\sum_{|k|>M}\|\chi_R(|D|)v_n^k(x)\|_{L^2_x},\\
&\lesssim\left(\sum_{|k|>M}\frac{1}{k^2}\right)^{1/2}\left(\sum_{|k|>M}k^2\|\chi_R(|\xi|)\hat{v}^k_n(\xi)\|_{L^2_\xi}^2\right)^{1/2},\\
&\lesssim\left(\sum_{|k|>M}\frac{1}{k^2}\right)^{1/2}\left(\sum_{|k|>M}k^2\|\hat{v}^k_n(\xi)\|_{L^2_\xi}^2\right)^{1/2},\\
&\lesssim \left(\sum_{|k|>M}\frac{1}{k^2}\right)^{1/2}\|v_n\|_{L^2_xH^1_y}\lesssim \left(\sum_{|k|>M}\frac{1}{k^2}\right)^{1/2}.
\end{aligned}
\end{equation*}

\noindent Since 
\begin{equation*}
\sum_{k=M+1}^\infty a_k\leq \int_{M}^\infty f(x)\,dx
\end{equation*}
where $f:\R\to\R^+$ is a decreasing function such that $a_k=f(k),$ (it is assumed here that $f(x)=x^{-2}$) then
\begin{equation*}
\left(\sum_{|k|>M+1} k^{-2}\right)^{1/2}\lesssim M^{-1/2}
\end{equation*}
and so 
\begin{equation*}
\|\chi^{>M}_R(|D|)v_n\|_{L^2_{x,y}}\lesssim M^{-1/2}.
\end{equation*}
The following interpolation result is a straightforward classical application of the H\"older inequality.
\begin{lemma}\label{int-lem} 
Let $p_1\leq p\leq p_2$ and $f\in L^{p_1}\cap L^{p_2}.$ Then $f\in L^p$ and given $\theta\in[0,1]$ such that $\frac1p=\frac{\theta}{p_1}+\frac{1-\theta}{p_2}$ the following estimates holds: 
\begin{equation*}
\|f\|_{L^p}\leq\|f\|_{L^{p_1}}^\theta\|f\|_{L^{p_2}}^{1-\theta}
\end{equation*}
\end{lemma}

\noindent Therefore by \autoref{int-lem} we have 
\begin{equation}\label{eq:2}
\begin{aligned}
\|\chi^{>M}_R(|D|)v_n\|_{L^{2q}_{x,y}}&\leq\|\chi^{>M}_R(|D|)v_n\|_{L^{2}_{x,y}}^{\theta}\|\chi^{>M}_R(|D|)v_n\|^{1-\theta}_{L^{2^*}_{x,y}},\\
&\lesssim\|\chi^{>M}_R(|D|)v_n\|_{L^{2}_{x,y}}^{\frac{(d+1)}{2q}-\frac{d-1}{2}}\lesssim M^{-\frac12\left(\frac{(d+1)}{2q}-\frac{d-1}{2}\right)}.
\end{aligned}
\end{equation}
\noindent It remains to estimate the term $\sum_{|k|\leq M}\chi_R(|D|)v_n^k(x)e^{iky}.$ Denoting by $D_M$ the Dirichlet Kernel $$D_M(y)=\sum_{k=-M}^Me^{iky},$$
we can write 
\begin{equation*}
\begin{aligned}
\chi_R^{\leq M}(|D|)v_n(x,y)=\sum_{|k|\leq M}\chi_R(|D|)v_n^k(x)e^{iky}&=\int_{\mathbb T}\chi_R(|D|)v_n(x,z)D_M(y-z)\,dz,\\
\end{aligned}
\end{equation*}
and we choose a sequence $(x_n,y_n)\in\R^d\times\mathbb T$ such that 
\begin{equation*}
\begin{aligned}
\|\chi_R^{\leq M}(|D|)v_n\|_{L^\infty_{x,y}}&\leq2\left|\chi_R^{\leq M}(|D|)v_n(x_n,y_n)\right|\\
&=2R^d\left|\int_{\R^d\times\mathbb T}\eta(Rx)D_M(y)v_n(x-x_n,y-y_n)\,dy\,dx\right|,
\end{aligned}
\end{equation*}
\noindent where $R^d\eta(Rx)=\mathcal F^{-1}(\chi_R(|\xi|)).$ Observe that $\eta(Rx)D_M(y)$ is a function in $L^2_{x,y}$ and that $\|\eta(Rx)D_M(y)\|_{L^2}\lesssim R^{-d/2}M\|\eta\|_{L^2}.$ Up to subsequences, from \eqref{def} and \eqref{lambda} we get 
\begin{equation*}
\begin{aligned}
\limsup_{n\to\infty}\|\chi_R^{\leq M}(|D|)v_n\|_{L^\infty_{x,y}}&\leq\limsup_{n\to\infty}2R^d\left|\int_{\R^d\times\mathbb T}\eta(Rx)D_M(y)v_n(x-x_n,y-y_n)\,dy\,dx\right|,\\
&=2R^d\left|\int_{\R^d\times\mathbb T}\eta(Rx)D_M(y)w(x,y)\,dy\,dx\right|,\\
&\leq2R^{d/2}M\lambda\|\eta\|_{L^2}\lesssim R^{d/2}M\lambda,
\end{aligned}
\end{equation*}
thus, again by interpolation, we infer that 
\begin{equation*}
\begin{aligned}
\|\chi_R^{\leq M}(|D|)v_n\|_{L^{2q}_{x,y}}\lesssim\|\chi_R^{\leq M}(|D|)v_n\|^{1-1/q}_{L^\infty_{x,y}}\|\chi_R^{\leq M}(|D|)v_n\|^{1/q}_{L^2_{x,y}}\lesssim \|\chi_R^{\leq M}(|D|)v_n\|^{1-1/q}_{L^\infty_{x,y}},
\end{aligned}
\end{equation*}
and then 
\begin{equation}\label{eq:3}
\begin{aligned}
\limsup_{n\to\infty}\|\chi_R^{\leq M}(|D|)v_n\|_{L^{2q}_{x,y}}\lesssim R^{\frac d2\left(\frac{q-1}{q}\right)}M^{\frac{q-1}{q}}\lambda^{\frac{q-1}{q}}.
\end{aligned}
\end{equation}
Combining \eqref{not-localiz},\eqref{eq:2} and \eqref{eq:3}, we obtain 
\begin{equation*}
\begin{aligned}
\limsup_{n\to\infty}\|v_n\|_{L^{2q}_{x,y}}\lesssim R^{\frac{q(d-1)-(d+1)}{2q}}+M^{\frac12\left(\frac{q(d-1)-(d+1)}{2q}\right)}+R^{\frac d2\left(\frac{q-1}{q}\right)}M^{\frac{q-1}{q}}\lambda^{\frac{q-1}{q}},
\end{aligned}
\end{equation*}
and by choosing $M\sim R^2$ we end up with 
\begin{equation*}
\begin{aligned}
\limsup_{n\to\infty}\|v_n\|_{L^{2q}_{x,y}}\lesssim R^{\frac{q(d-1)-(d+1)}{2q}}+\left(R^{\frac{d+4}{2}}\lambda\right)^{\frac{q-1}{q}}.
\end{aligned}
\end{equation*}
We now consider as radius $R=\lambda^\beta,$ and so
\begin{equation*}
\begin{aligned}
\limsup_{n\to\infty}\|v_n\|_{L^{2q}_{x,y}}\lesssim \lambda^{ \beta \left(\frac{q(d-1)-(d+1)}{2q}\right)}+\left(\lambda^{\beta\left(\frac{d+4}{2}\right)+1}\right)^{\frac{q-1}{q}}.
\end{aligned}
\end{equation*}
Defining now $\beta$ in this way:
\begin{equation*}
\begin{aligned}
 &\beta \left(\frac{q(d-1)-(d+1)}{2q}\right)=\frac{q-1}{q}\left(\beta\left(\frac{d+4}{2}\right)+1\right),\\
 &\iff\beta\left(\frac{q(d-1)-(d+1)}{2q}-\frac{q-1}{q}\frac{d+4}{2}\right)=\frac{q-1}{q},\\
 &\iff \beta(q(d-1)-(d+1)-(q-1)(d+4))=2(q-1),\\
 &\iff \beta(3-5q)=2(q-1)\iff \beta=\frac{2(q-1)}{3-5q},
\end{aligned}
\end{equation*}
we observe that $\beta=\frac{2(q-1)}{3-5q}<0,$ and conclude with 
\begin{equation*}
\begin{aligned}
\limsup_{n\to\infty}\|v_n\|_{L^{2q}_{x,y}}\lesssim \lambda^{\frac{q-1}{3-5q}\left(\frac{q(d-1)-(d+1)}{q}\right)}.
\end{aligned}
\end{equation*}
\end{proof}


\begin{remark} We notice that $w$ actually belongs to  $H^1_{x,y}$ since the weak limit clearly enjoys this regularity.
\end{remark}
\vspace{2mm}
We now fix some notations 
used in the following part. We define with $v(t,x,y)$ or simply $v(t)$ the free evolution with respect to the linear Klein-Gordon equation, with Cauchy datum $\vec v^0=(v_0,v_1)$ and we define by $\vec v(t)=e^{tH}\vec v^0=(v(t),\partial_tv(t))^T,$ where $e^{tH}$ has been introduced in \eqref{exp-matrix}. Then, we give the following decomposition for a time-independent bounded sequence in $H^1\times L^2.$ We first introduce the following lemma which will be useful after.\\

To shorten notation, we write from now on $\mathcal H=H^1\times L^2.$ \\

\begin{lemma}\label{lem:4.6}
Let $\vec f_n\rightharpoonup0$ in $\mathcal H$. Then we have:
\begin{itemize}
\item\label{cond2.15} $t_n\ra\bar t\in \R \implies e^{t_nH} \vec f_n(x,y)\rightharpoonup0$ in $\mathcal H,$
\item\label{cond2.14} $e^{(t^2_n-t^1_n)H} \vec f_n(x-(x_n^1-x_n^2),y)\rightharpoonup \vec g\neq0\implies |t_n^2-t_n^1|+|x_n^2-x_n^1|\rightarrow+\infty.$
\end{itemize} 
\end{lemma}
\begin{proof}
For the first point, we make use of the continuity property of the propagator: by denoting by $(\cdot, \cdot)_{\mathcal H}$ the scalar product in $\mathcal H,$ for any $\vec\psi\in\mathcal H$ we have 
\begin{equation*}
\begin{aligned}
(e^{t_nH}\vec f_n,\vec\psi)_{\mathcal H}&=(e^{t_nH}\vec f_n-e^{\bar{t}H}\vec f_n,\vec\psi)+(\vec f_n,e^{-\bar{t}H}\vec\psi)_{\mathcal H},\\
&=(\vec f_n,e^{-t_nH}\vec\psi-e^{\bar tH}\vec\psi)_{\mathcal H}+(\vec f_n(x,y),e^{-\bar{t}H}\vec\psi)_{\mathcal H},\\
&=(\vec f_n,e^{-t_nH}\vec\psi-e^{\bar tH}\vec\psi)_{\mathcal H}+o(1).
\end{aligned}
\end{equation*}
The conclusion follows, up to subsequences, since it holds in the $L^2\times L^2$ topology by exploiting the continuity of the flow.

The second point is proved in its contrapositive form. Suppose  that $s_n:=(t_n^2-t_n^1)$ and $\xi_n:=(x_n^2-x_n^1)$ are bounded. Then, up to subsequences, $s_n\ra s\in \R$ and $\xi_n\ra\bar\xi\in\R^d.$ We prove that  $e^{s_nH} \vec f_n(x-\xi_n,y)\rightharpoonup0$ in $\mathcal H.$ But as before
\begin{equation}
\begin{aligned}
(e^{s_nH} \vec f_n(x-\xi_n,y),\vec\psi)_{\mathcal H}&=(e^{sH} \vec f_n(x,y),\vec\psi(x+\xi,y))_{\mathcal H}+o(1)=(\vec f_n(x,y),e^{-sH}\vec\psi(x+\xi,y))_{\mathcal H}+o(1)\rightharpoonup0.
\end{aligned}
\end{equation}
\end{proof}
We can now state the following result, whose iteration will give the Profile Decomposition Theorem. 
\begin{prop}\label{preliminary-prop}
Let $\{\vec v_n^0\}_n$ be a bounded sequence in $\mathcal H$ and $1\leq d\leq4.$ Then, for suitable sequences $\{t_n\}_{n\in\mathbb N}\subset\mathbb R,\{x_n\}_{n\in\mathbb N}\subset\mathbb R^d,$ possibly after extractions of subsequences (still denoted with the subscript $n$), we can write, for every $n\in\mathbb N$
\begin{equation*}
\vec v_n(-t_n,x-x_n,y)=\vec\psi(x,y)+\vec W_n(x,y)
\end{equation*}
where $\vec v_{n}(t,x,y)=e^{tH}\vec v_n^0$ and where the components of $\vec\psi$ are denoted by $(\psi,\partial\psi).$ Moreover, the following properties hold:
\begin{equation*}
\vec W_n\overset{n\to\infty}{\rightharpoonup}0 \quad \textit{in} \quad \mathcal H,
\end{equation*}
\begin{equation}\label{eq:small}
\limsup_{n\ra\infty}\|v_n(t,x,y)\|_{L^{\infty}L^q}\lesssim\|\psi\|_{L^2}^{e}\quad\text{for\,\,any}\quad q\in(2,2^*),
\end{equation}
where $e>0$ is given in \autoref{preli} and as $n\to\infty$ 
\begin{equation}\label{energy-ortho}
\|\vec v_n^0\|^2_{\mathcal H}=\|\vec\psi\|^2_{\mathcal H}+\|\vec W_n\|^2_{\mathcal H}+o(1).
\end{equation}
\noindent Similarly, for the $L^{\a+2}$ norm, as $n\to\infty,$ we have 
\begin{equation}\label{Lp:ortho}
\| v_n^0\|^{\a+2}_{L^{\a+2}}=\|\psi\|^{\a+2}_{L^{\a+2}}+\| W_n\|^{\a+2}_{L^{\a+2}}+o(1).
\end{equation}
Furthermore, the translation sequences $\{t_n\}_{n\in\mathbb N}$ and $\{x_n\}_{n\in\mathbb N}$ satisfy the dichotomies below:
\begin{equation}\label{dicoth}
\begin{aligned}
\textit{either}\quad t_n&=0, \qquad\textit{or}\quad t_n\rightarrow\pm\infty;\\
\textit{either}\quad x_n&=0, \qquad\textit{or}\quad |x_n|\rightarrow\infty.
\end{aligned}
\end{equation}
\end{prop}

\begin{proof}
Define $\vec v_n(t,x,y):=e^{tH}\vec v_n^0,$ namely $\vec v_n(t)$ is the linear evolution of $\vec v_n^0$ by the linear Klein-Gordon flow. Since the energy is preserved along the flow, the sequence $\vec v_n(t)$ is bounded in $L^\infty_t\mathcal H$ and by Sobolev embedding the sequence $\{v_n(t)\}_{n\in\mathbb N}$ is bounded in  $L^\infty L^q$ norm, for any $q\in(2,2^*)$. Thus, let us now choose a sequence of times $\{t_n\}_{n\in\mathbb N}$ such that 
\begin{equation}\label{eq:3.20}
\|v_n(-t_n,x,y)\|_{L^q}>\frac12\|v_n(\cdot,x,y)\|_{L^\infty L^q}.
\end{equation}
In the spirit of previous lemma, we consider $\Lambda\left(v_n(-t_n,x,y)\right)$ and $\lambda\left(v_n(-t_n,x,y)\right).$ Let $\{(x_n,y_n)\}_{n\in\mathbb N}$ be a sequence in $\mathbb R^d\times \mathbb T$ and $\vec\psi(x,y)=(\psi,\partial\psi)(x,y)\in\mathcal H$ be  such that, up to subsequences, 
\begin{equation*}
\vec v_n(-t_n,x-x_n,y-y_n)\rightharpoonup\vec\psi 
\end{equation*}
in $\mathcal H$ as $n\to\infty$. Then we get 
\begin{equation}\label{eq:3.21}
\vec v_n(-t_n,x-x_n,y-y_n)=\vec\psi+\vec W_n,\qquad \vec W_n\rightharpoonup0,
\end{equation}
the latter weak convergence occurring in $\mathcal H$ and in addition
\begin{equation}\label{eq:3.22}
\lambda\left(v_n(-t_n,x,y)\right)\lesssim\|\psi\|_{L^2}.
\end{equation}
The relation \eqref{eq:3.22} along with \autoref{preli} implies that 
\begin{equation*}
\limsup_{n\to\infty}\|v_n(-t_n)\|_{L^q}\lesssim\|\psi\|_{L^2}^e\quad \text{for\,\, any}\quad q\in(2,2^*),
\end{equation*}
and then \eqref{eq:small} follows by \eqref{eq:3.20}.\\ 

\noindent By definition, from \eqref{eq:3.21} we can write
\begin{equation}\label{eq:3.23}
\begin{aligned}
\vec {v}_n^0(x,y)=e^{t_nH}\vec\psi(x+x_n,y+y_n)+e^{t_nH}\vec W_n(x+x_n,y+y_n),
\end{aligned}
\end{equation}
and since $e^{tH}$ is an isometry on $\mathcal H$ and its adjoint is given by $e^{-tH}$, together with the fact that $\vec W_n\rightharpoonup0,$ we get, as $n\rightarrow\infty,$  
\begin{equation*}
\|\vec v_n^0\|_{\mathcal H}^2=\|\vec\psi\|_{\mathcal H}^2+\|\vec W_n\|_{\mathcal H}^2+o(1).
\end{equation*}

We pursue the proof by showing the orthogonality property of the potential energy, by distinguishing three cases. In the following, the Lebesgue exponent $\a+2,$ is given by the same $\a$  appearing in the nonlinearity of \eqref{NLKG}.
\\

\noindent\emph{Case 1: $|t_n|\ra\infty.$} From \eqref{eq:3.23} we see that \eqref{Lp:ortho} holds, observing that $W_n$ is uniformly bounded and using the dispersive estimate \eqref{disp:m3} and a density argument; hence the orthogonality in $L^{\a+2}$.\\

\noindent Since $\{y_n\}_{n\in\mathbb N}\subset \mathbb T$ which is compact, in the next two cases we can assume that up to subsequence $y_n\to\bar y\in \mathbb T.$\\

\noindent\emph{Case 2: $t_n\ra\bar t$ \& $x_n\ra\bar x.$} We claim the following: 
\begin{equation*}
\vec {v}_n^0(x,y)-e^{t_nH}\vec\psi(x+x_n,y+y_n)=e^{t_nH}\vec W_n(x+x_n,y+y_n)\ra 0,
\end{equation*}
for almost every $(x,y)\in\R^d\times\mathbb T.$
In fact $$(e^{t_nH} \vec W_n(x+x_n,y+y_n),\vec\psi)_{\mathcal H}=(\vec W_n,e^{-\bar{t}H}\vec \psi(x-\bar x,y-\bar y))_{\mathcal H}+o(1)=o(1),$$ if we localize in the euclidean part, i.e. if we consider the restriction of $e^{t_nH}\vec W_n(x+x_n,y+y_n)$ on a compact set $K\subset\R^d$. The compactness of $K\times \mathbb T$ gives by the Rellich-Kondrakhov theorem, that  $W_n(t_n,x+x_n,y+y_n)$ strongly converges towards zero in $L^p(K\times\mathbb T)$ for any $p\in(2,2^*),$ see \cite{Heb1}. Therefore we have $(x,y)$-almost everywhere convergence towards zero of $W_n(t_n,x+x_n,y+y_n).$ We recall that the Brezis-Lieb Lemma (see \cite{BL}) holds on a general measured space, therefore the same argument given in \cite{BV} yields to the $L^{\a+2}$ orthogonality in the case $t_n\ra\bar t$ and $x_n\ra \bar x.$ \\

\noindent\emph{Case 3: $t_n\ra\bar t$ \& $|x_n|\ra\infty.$} Similar arguments apply to the remaining situation $t_n\ra\bar t$ and $|x_n|\ra\infty.$
\\

It remains to prove that we can rearrange the sequences of translation parameters $\{t_n\}_{n\in\mathbb N},$ $\{x_n\}_{n\in\mathbb N}$ and $\{y_n\}_{n\in\mathbb N}$. Namely, we wish to have that for any $n\in\mathbb N,$ $t_n=0$ or $t_n\ra\pm\infty$, and similarly for  $\{x_n\}_{n\in\mathbb N}$, while $y_n$ can be assumed to be trivial. In the following, by $t_n\ra\bar t$ and  $x_n\ra\bar x$ we will implicitly assume that this possibly holds after extraction of subsequences from bounded sequences.   
\\

\noindent\emph{Case 1: $t_n\ra\bar t$ \& $|x_n|\ra\infty.$} By continuity of the linear flow 
\begin{equation*}
e^{t_nH}\vec\psi\overset{\mathcal H}{\longrightarrow}\vec\phi, \qquad \vec\phi(x,y):=e^{\bar{t}H}\vec\psi(x,y).
\end{equation*}
We rewrite $\vec v_n^0$ as 
\begin{equation*}
\vec v_n^0(x-x_n,y-y_n)=\vec\phi(x,y)+e^{t_nH}\vec W_n(x,y)+\vec r_n(x,y)=\vec\phi(x,y)+\vec\rho_n(x,y),
\end{equation*}
where $\vec r_n\rightarrow0$ strongly in $\mathcal H$ and $\vec\rho_n=e^{t_nH}\vec W_n(x,y)+\vec r_n(x,y).$
From \autoref{lem:4.6} it follows that if $\vec h_n\rightharpoonup0$ in $\mathcal H$ and $t_n\rightarrow\bar t$ then $e^{t_nH}\vec h_n\rightharpoonup0.$ Therefore $\vec\rho_n\rightharpoonup0$ in $\mathcal H$. Its is true whether $x_n\rightarrow x_0\in\R^d$ or $|x_n|\rightarrow\infty.$ Translating the profiles by $\bar y,$ namely by choosing  $\vec\phi(x,y):=e^{\bar{t}H}\vec\psi(x,y-\bar y)$ we can also assume that $y_n=0.$
\noindent\emph{Case 2: $t_n\ra\bar t$ \& $x_n\ra\bar x.$} If $t_n\rightarrow\bar t\in\R$ and also $x_n\rightarrow\bar x\in\R^d$ we proceed similarly by adding a space translation: namely as before but considering $\vec\phi:=e^{\bar tH}\vec\psi(x-\bar x,y-\bar y).$\\ 

\noindent\emph{Case 3: $t_n\ra\pm\infty$ \& $x_n\ra\bar x.$} If $t_n\ra\pm\infty$ and $x_n\rightarrow\bar x\in\R$ then we change the function by translating in the space variables only, i.e. we consider $\vec\phi:=\vec\psi(x-\bar x,y-\bar y).$\\

\noindent\emph{Case 4: $|t_n|\ra\infty$ \& $|x_n|\ra\infty.$} By extracting subsequences we have the desired property, again by translating the profiles in the $y$-variable only.
\end{proof}

We can now state the linear profile decomposition for a bounded sequence of linear solutions in the energy space.  
\begin{theorem}[Linear Profile Decomposition] \label{lpd}Let $\{\vec u_n(t,x,y)\}_{n\in\mathbb N}$ be a sequence of solutions to the linear Klein-Gordon equation, bounded in $H^1(\R^d\times\mathbb T)\times L^2(\R^d\times\mathbb T)$ for $1\leq d\leq4.$ Recall that $\| \vec u_n(t,x,y)\|_{\mathcal H}=\| \vec u_n(0,x,y)\|_{\mathcal H},$ thus we are assuming that $\sup_n\|\vec u_n(0)\|_{\mathcal H}<\infty.$ For any integer $k>1$ the decomposition below holds: 
\begin{equation*}
 \vec u_n(t,x,y)=\sum_{1\leq j<k} \vec v^j(t-t^j_n,x-x_n^j,y)+ \vec R^k_n(t,x,y),
\end{equation*}
where $\vec v^j$ are solutions to linear Klein-Gordon with suitable initial data and the translation sequences satisfy 
\begin{equation*}
\lim_{n\to\infty}\left(|t_n^k-t_n^j|+|x_n^k-x_n^j|\right)=\infty, \qquad \forall\,j\neq k.
\end{equation*}
along with the same dichotomy property of \eqref{dicoth}.
Moreover, for $q\in(2,2^*)$
\begin{equation*}
\lim_{k\ra\infty}\limsup_{n\ra\infty}\|R_n^k\|_{L^\infty L^q}=0
\end{equation*}
which in turn implies that 
\begin{equation*}
\lim_{k\ra\infty}\limsup_{n\ra\infty}\|R_n^k\|_{L^{\a+1}L^{2(\a+1)}}=0.
\end{equation*} 
Furthermore as $n\rightarrow\infty,$
\begin{equation*}
\| \vec u_n(0,x,y)\|^2_{\mathcal H}=\sum_{1\leq j<k}\| \vec v^j_n\|^2_\mathcal H+\| \vec R^k_n\|^2_{\mathcal H}+o(1),
\end{equation*}
and
\begin{equation*}
\| u_n(0,x,y)\|^{\a+2}_{L^{\a+2}}=\sum_{1\leq j<k}\| v^j_n\|^{\a+2}_{L^{\a+2}}+\| R^k_n\|^{\a+2}_{L^{\a+2}}+o(1).
\end{equation*}
\end{theorem}

\begin{proof}
We iterate several times the result of \autoref{preliminary-prop}. We consider $\{\vec v_n\}_{n\in\mathbb N}$ as the sequence of initial data of the linear solution $\{\vec u_n(t,x,y)\}_{n\in\mathbb N}$; namely we consider the sequence $\{\vec u_n(0,x,y)\}_{n\in\mathbb N}$ as a bounded sequence in $\mathcal H$.
Let $\{t^1_n\}_{n\in\mathbb N}$ be the sequence given in the proposition above and $\{x^1_n\}_{n\in\mathbb N}\subset\mathbb R^d$ be such that, up to subsequences,
\begin{equation*}
\vec u_n(-t^1_n,x-x^1_n,y)\rightharpoonup \vec{\psi}^1(x,y) 
\end{equation*}
in $\mathcal H$. Then
\begin{equation*}
\vec u_n(-t^1_n,x-x^1_n,y)= \vec\psi^1(x,y)+ \vec W^1_n(x,y),
\end{equation*}
with $\vec W^1_n\rightharpoonup0$ in $\mathcal H$. 
It follows, as $n\to\infty,$ that 
\begin{equation*}
\vec u_n(0,x,y)=e^{t^1_nH}\vec\psi^1(x+x^1_n,y)+e^{t^1_nH}\vec W^1_n(x+x^1_n,y):=e^{t^1_nH}\vec\psi^1(x+x^1_n,y)+ \vec R_n^1(x,y),
\end{equation*}
where 
\begin{equation}\label{eq:wc1}
e^{-t^1_nH} \vec R^1_n(x-x^1_n,y)= \vec W^1_n(x,y)\rightharpoonup0,
\end{equation}
in $\mathcal H,$ and that 
\begin{equation*}
\|\vec u_n(0)\|^2_{\mathcal H}=\|\vec\psi^1\|^2_{\mathcal H}+\|\vec R^1_n\|^2_{\mathcal H}+o(1)=\|\vec\psi^1\|^2_{\mathcal H}+\|\vec W^1_n\|^2_{\mathcal H}+o(1).
\end{equation*}
Similar claim can be proved for the $L^{\a+2}-$norm (potential energy).
We now consider the functions $\vec R^1_n(x,y)=e^{t^1_nH}\vec W^1_n(x+x^1_n,y)$ as bounded sequence in $\mathcal H$. As before, we can write 
\begin{equation*}
\vec R^1_n(x,y)=e^{t^2_nH} \vec\psi^2(x+x^2_n,y)+e^{t^2_nH}\vec W^2_n(x+x^2_n,y):=e^{t^2_nH}\vec\psi^2(x+x^2_n,y)+\vec R^2_n(x,y),
\end{equation*}
where $\vec W^2_n\rightharpoonup 0$ in $\mathcal H$ and 
\begin{equation*}
\|\vec R^1_n\|_{\mathcal H}^2=\|\vec\psi^2\|^2_{\mathcal H}+\|\vec R^2_n\|^2_{\mathcal H}+o(1)=\|\vec\psi^2\|^2_{\mathcal H}+\|\vec W^2_n\|^2_{\mathcal H}+o(1).
\end{equation*}
It implies that at the second step we have
\begin{equation*}
\vec u_n(0,x,y)=e^{t^1_nH}\vec\psi^1(x+x^1_n,y)+e^{t^2_nH}\vec\psi^2(x+x^2_n,y)+\vec R^2_n(x,y), 
\end{equation*}
and by ``applying'' the linear propagator on both sides we get 
\begin{equation*}
\vec u_n(t,x,y)=e^{(t+t^1_n)H}\vec\psi^1(x+x^1_n,y)+e^{(t+t^2_n)H} \vec\psi^2(x+x^2_n,y)+e^{tH}\vec R^2_n(x,y).
\end{equation*}
Moreover, as $n\to\infty,$
\begin{equation*}
\|\vec u(t,x,y)_n\|_{\mathcal H}^2=\|\vec\psi^1\|^2_{\mathcal H}+\|\vec\psi^2\|^2_{\mathcal H}+\|\vec R^2_n\|^2_{\mathcal H}+o(1)=\|\vec\psi^1\|^2_{\mathcal H}+\|\vec\psi^2\|^2_{\mathcal H}+\|\vec W^2_n\|^2_{\mathcal H}+o(1),
\end{equation*}
and the orthogonality for the $L^{\a+2}-$norm can be similarly proved.
Recall that 
\begin{equation*}
\begin{aligned}
e^{t^1_nH}\vec W^1_n(x+x^1_n,y)=\vec R^1_n(x,y)
=e^{t^2_nH} \vec\psi^2(x+x^2_n,y)+e^{t^2_nH} \vec W^2_n(x+x^2_n,y),
\end{aligned}
\end{equation*}
and so
\begin{equation*} 
\begin{aligned}
e^{(t^1_n-t^2_n)H}\vec W^1_n(x+(x^1_n-x^2_n),y)&=\vec\psi^2(x,y)+\vec W^2_n(x,y),
\end{aligned}
\end{equation*}
with $ \vec W^2_n\rightharpoonup 0$ in $\mathcal H,$ and this implies the weak convergence in $\mathcal H$
\begin{equation*}
e^{(t^1_n-t^2_n)H}\vec W^1_n(x+(x^1_n-x^2_n),y)\rightharpoonup\vec\psi^2(x,y).
\end{equation*}
\autoref{lem:4.6}, which is the equivalent of \cite{BV}*{Lemma 2.1} in our context, allows us to conclude with the orthogonality condition
\begin{equation*}
|t^1_n-t^2_n|+|x^1_n-x_n^2|\rightarrow\infty.
\end{equation*}

\noindent Iterating this construction we end up, at the $k^{th}$ step, with
\begin{equation*}
\begin{aligned}
\vec u_n(t,x,y)&=e^{(t+t^1_n)H} \vec\psi^1(x+x^1_n,y)+\dots+e^{(t+t^{k-1}_n)H}\vec\psi^{k-1}(x+x^{k-1}_n,y)+e^{tH}\vec R^k_n(x,y), 
\end{aligned}
\end{equation*}
where 
\begin{equation*}
\vec R^k_n(x,y)=e^{t^k_nH}\vec W^k_n(x+x^k_n,y),\qquad W^k_n\rightharpoonup0.
\end{equation*}

\noindent Moreover the free energy orthogonality holds:
\begin{equation*}
\|\vec u_n(t,x,y)\|_{\mathcal H}^2=\|\vec\psi^1\|_{\mathcal H}^2+\dots+\|\vec\psi^{k-1}\|_{\mathcal H}^2+\|\vec R^k_n\|_{\mathcal H}^2,
\end{equation*}
and by the fact that the l.h.s. is uniformly bounded in $L^\infty_t \mathcal H$ we get
\begin{equation*}
\lim_{k\to\infty}\|\psi^k\|_{L^2}\leq\lim_{k\to\infty}\|\psi^k\|_{H^1}\leq\lim_{k\rightarrow\infty}\|\vec\psi^k\|_{\mathcal H}=0.
\end{equation*}
Using \eqref{eq:3.22} we obtain the smallness of the remainders in the sense of 
\begin{equation*}
\limsup_{k\rightarrow\infty}\limsup_{n\rightarrow\infty}\|R^k_n\|_{L^\infty L^q}=0.
\end{equation*}
The proof of the smallness in the Strichartz norm
\begin{equation*}
\lim_{k\ra\infty}\limsup_{n\ra\infty}\|R_n^k\|_{L^{\a+1}L^{2(\a+1)}}=0,
\end{equation*}
is done by interpolation (see \autoref{app:lemma1} for all detailed computations). The proof of \autoref{lpd} is complete.
\end{proof}

\section{The critical element}\label{sec:critical element}

\autoref{lpd} is the key tool for the construction of a minimal (with respect to its energy) non-scattering solution to \eqref{NLKG} with some compactness property. We define the following critical energy:

\begin{equation*}
\begin{aligned}
E_c=\sup\{E>0\,|\, &(f,g)\in\mathcal H\quad\textit{and}\quad E(f,g)<E\\
&\implies u_{(f,g)}(t)\in L^{\a+1}L^{2(\a+1)}<\infty\}
\end{aligned}
\end{equation*}
\noindent where $u_{(f,g)}(t)$ denotes here the global solution to \eqref{NLKG} with Cauchy data $f$ and $g.$ Our final aim, at the end of the paper, is to exclude that $E_c$ is finite. 

The result stated in \autoref{thm:small} ensures that $E_c>0.$ The strategy consists in a contradiction argument. If we suppose that $E_c$ is finite, we will show that there exists a critical solution $u_c$ to \eqref{NLKG}, with energy $E_c$, such that it does not belong to the Strichartz space $L^{\a+1}L^{2(\a+1)}$. It will moreover enjoy some compactness properties. The latter will imply that such critical solution must be the trivial one, hence a contradiction. 

We first proceed with the construction of the critical solution, based on the profile decomposition theorem \autoref{lpd}.
We remark that it is a linear statement. Since we are dealing with a nonlinear equation, we give the following perturbation lemma which will enable us to absorb the nonlinear terms in the remainders of the profile decomposition theorem, in a proper way. 
The following long time perturbation theorem is contained in \cite{NS11} where it is proved for cubic focusing NLKG on $\mathbb R^3$, namely for $\a=2,$ with the opposite sign in front of the nonlinearity and in an euclidean framework. We report here the statement modified to fit our setting, for the sake of completeness.
\begin{lemma}\label{pert-lemma}
For any $M>0$ there exist $\varepsilon=\varepsilon(M)>0$ (possibly very small) and $c=c(M)>0$ (possibly very large) such that the following fact holds. Fix $t_0\in \R$ and suppose that 
\begin{equation*}
\begin{aligned}
&\|v\|_{L^{\a+1}L^{2(\a+1)}}\leq M\\
\|e_u\|_{L^{1}L^{2}}+\|e_v\|_{L^{1}L^{2}}&+\|w_0\|_{L^{\a+1}L^{2(\a+1)}}\leq\varepsilon^\prime\leq\varepsilon(M)
\end{aligned}
\end{equation*}
where $u,v\in\bigcap_{h=0}^1\mathcal C^{h}(\R;H^{1-h}),$ $e_z=\partial_{tt}z-\Delta_{x,y}z+z+|z|^\a z$ and $\vec w_0(t)=e^{(t-t_0)H}(\vec u(t_0)-\vec v(t_0)).$ Then 
\begin{equation*}
\begin{aligned}
\|u\|_{L^{\a+1}L^{2(\a+1)}}&<\i,\\
\|\vec u-\vec v-\vec w_0\|_{L^{\i}\mathcal H}+\|u-v\|_{L^{\a+1}L^{2(\a+1)}}&\leq c(M)\eps^\prime.
\end{aligned}
\end{equation*}
\end{lemma}
\textit{Sketch of the proof: }
This follows by the same proof as in \cite{NS11} (where the nonlinearity was $u^2u$), but with the following inequality to estimate the nonlinear part: 
\begin{equation*}
||u+v|^\a(u+v)-|u|^\a u|\leq C(|u|^\a+|v|^\a)|v|=C(|u|^\a|v|+|v|^{\a+1}).
\end{equation*}

Once every ingredient is given, we continue with the extraction of the critical solution. We therefore assume that $E_c<\infty.$ Let $(f_n,g_n)\in\mathcal H$ be a sequence of Cauchy data such that $E(f_n,g_n)\ra E_c$ as $n\ra+\infty$ and let $u_n(t):=u_{(f_n,g_n)}(t)$ be the corresponding solutions to \eqref{NLKG} which exist globally in time but do not belong to $L^{\a+1}L^{2(\a+1)},$ i.e. $\|u_n\|_{L^{\a+1}L^{2(\a+1)}}=\infty.$ The last condition means that we are considering a maximizing sequence $(f_n,g_n)\in\mathcal H$ whose corresponding solutions do not satisfy the scattering property. 

Since $E(f_n,g_n)\ra E_c$ and the energy is a conserved quantity, we can state that $\vec u^0_n:=(f_n,g_n)$ is uniformly bounded in $\mathcal H$. And since the Klein-Gordon linear flow preserves the $\mathcal H$ norm, the sequence $e^{tH}\vec u_n^0$ is uniformly bounded in $L^\infty\mathcal H.$ Thus we can apply the linear profile decomposition to this sequence of free solutions and we can write
\begin{equation}\label{eq:4.51}
e^{tH}\vec u_n^0=\sum_{1\leq j<k} \vec v^j_n(t)+ \vec R^k_n(t,x,y),
\end{equation}
where $\vec v^j_n(t)=\vec v^j(t-t^j_n,x-x_n^j,y)=e^{(t-t^j_n)H}\vec\psi^j(x-x^j_n,y)$ for suitable $\vec \psi^j\in\mathcal H.$ 
We recall that the profile decomposition theorem given above ensures the orthogonality of the translation sequences in the sense of  
\begin{equation}\label{eq:6.3}
\lim_{n\ra+\infty}\left(|t_n^h-t_n^j|+|x_n^h-x_n^j|\right)=+\infty,
\end{equation}
\noindent for all $j\neq h$, the smallness of the remainders in the sense of 
\begin{equation}\label{eq:4.61}
\lim_{k\ra\infty}\limsup_{n\ra\infty}\|R_n^k(t)\|_{L^\infty L^q\cap L^{\a+1}L^{2(\a+1)}}=0,
\end{equation}
as well as the pythagorean expansions of the quadratic and super quadratic terms of the energy. More precisely, for $n\rightarrow\infty,$  
\begin{equation}\label{eq:6.5}
\|(f_n,g_n)\|_{\mathcal H}^2=\| \vec u_n(0,x,y)\|^2_{\mathcal H}=\| \vec u_n(t,x,y)\|^2_{\mathcal H}=\sum_{1\leq j<k}\| \vec v^j_n\|^2_\mathcal H+\| \vec R^k_n\|^2_{\mathcal H}+o(1),
\end{equation}
and 
\begin{equation}\label{eq:6.51}
\| u_n(0,x,y)\|^{\a+2}_{L^{\a+2}}=\sum_{1\leq j<k}\| v^j_n\|^{\a+2}_{L^{\a+2}}+\| R^k_n\|^{\a+2}_{L^{\a+2}}+o(1).
\end{equation}
We suppose that $k>1$ and we follow the same strategy as in \cite{BV, NS11}. We have that, at most, in one case we can have that both space and time translations sequences are trivial, due to \eqref{eq:6.3}. Without loss of generality we can suppose that this case happens when $j=1$, and since we are assuming $k>1$ we have, by orthogonality of the energy expressed by summing up \eqref{eq:6.5} and \eqref{eq:6.51}, that $\vec\psi^1$ is 
such that the corresponding solution $z^1:=u_{\vec\psi^1}$ to \eqref{NLKG} scatters, as it belongs to $L^{\a+1}L^{2(\a+1)}$ by definition. In the other cases $j\geq2,$ we associate to a linear profile $\vec\psi^j,$  a nonlinear profile in a proper way. 
We associate a nonlinear profile $V^j$ to each linear profile $v^j$ thanks to the following procedure: $V^j$ is a nonlinear solution to \eqref{NLKG} such that \begin{equation*}
\lim_{n\to\infty}\|\vec v^j(t^j_n)-\vec V^j(t^j_n)\|_{\mathcal H}=0.
\end{equation*}
Recall that by the dichotomy property of the parameters, for every $j$, $\lim_{n\to\infty}t^j_n=0$ or $\lim_{n\to\infty}|t^j_n|=\infty.$ Then $V^j$ is locally defined both in a neighborhood of $t=0$ or $|t|=\infty:$ the first property follows by the local well-posedness theory, while the second one by the existence of the wave operators. Due to the defocusing nature of the equation, $V^j$ is actually globally defined. Orthogonality of the energy given by \eqref{eq:6.5} together with \eqref{eq:6.51} implies that any nonlinear profile $V^j$ has an energy less than the minimal one $E_c.$

\noindent Let us define $$V(t)=\sum_{j=1}^kV^j(t-t^j_n,x-x^j_n,y);$$ we use the perturbation lemma with $V$ instead of $v$ in \autoref{pert-lemma} and $u_n$ in the role of $u$ of the \autoref{pert-lemma}. As in \cite{NS11} this would imply that 
\begin{equation*}
\limsup_{n\to\infty}\|\sum_{j=1}^kV^j(t-t^j_n,x-x^j_n,y)\|_{L^{\a+1}L^{2(\a+1)}}<C<\infty,\quad\text{ uniformly in }\; k,
\end{equation*}
and \autoref{pert-lemma} gives
$$\limsup_{n\to\infty}\|u_n\|_{L^{\a+1}L^{2(\a+1)}}<C,$$
which is a contradiction. Therefore $k=1$, and the precompactness of the trajectory up to a translation also follows by \cite{NS11}. We can summarize the core result of this section in the following theorem. 

\begin{theorem}\label{exist:min}
There exists an initial datum $(f_c,g_c)\in H^1(\R^d\times \mathbb T)\times L^2(\R^d\times \mathbb T)$ such that the corresponding solution $u_c(t)$ to \eqref{NLKG} is global and $\|u_c\|_{L^{\a+1}L^{2(\a+1)}}=\infty.$ Moreover there exists a path $x(t)\in\R^d$ such that $\{u_c(t,x-x(t),y),\partial_t u_c(t,x-x(t),y), t\in\R^+\}$ is precompact in $H^1(\R^d\times \mathbb T)\times L^2(\R^d\times \mathbb T).$ 
\end{theorem}

\section{Rigidity}\label{sec:rigidity}
This section establishes that the minimal element built in the previous section cannot exist. The first step is to prove the validity of the finite propagation speed in our framework. It will be useful to control the growth of the translation path $x(t)\in\R^d$ given in \autoref{exist:min}.
Let us first recall this simple result.
\begin{lemma}
Let $f$ be smooth and $B(x_0,r)\subset \R^d$ the ball centered in $x_0$ with radius $r$. The following equality holds:
\begin{equation*}
\frac{d}{dr}\int_{B(x_0,r)}f(x)\,dx=\int_{\partial B(x_0,r)}f(\sigma)\,d\sigma,
\end{equation*}
where $\partial B(x_0,r)$ is the boundary of $B(x_0,r)$ and $d\sigma$ is the surface measure on $\partial B(x_0,r)$. 
\end{lemma}
\begin{proof}
The proof is straightforward once switched in radial coordinates. 
\end{proof}
We then state the following, which is the finite time propagation speed mentioned above. The notation $B(x_0,r)^c$ stands for $\R^d\setminus B(x_0,r).$
\begin{prop}
Let $u$ be the solution to \eqref{NLKG} with Cauchy datum $(u_0,u_1)$ vanishing on $B(x_0,r)^c\times\mathbb T,$ for some $r>0$. Then $\vec u(t)=(u,\partial_tu)(t)$ vanishes on $K(x_0,r):=\{t\geq0,\, x\in B(x_0,r+t)^c,\, y\in\mathbb T\}.$
\end{prop}
\begin{proof} Fix $r>0,$ $x_0\in\R^d,$ consider the balls $B(x_0,r+t):=B(t+r)$ and define the local energy $E_r(t)$ as 
\begin{equation*}
E_r(t)=\frac12\int_{\mathbb T}\int_{B(r+t)}\left(|\partial_tu|^2+|\nabla u|^2+|u|^2+\frac{2}{\a+2}|u|^{\a+2}\right)(t)\,dx\,dy.
\end{equation*}
Assume that $u(t,x,y)$ is smooth enough (by a classical regularization argument, the following then extends to rougher solutions), and let us calculate the first time derivative of the local energy:
\begin{equation*}
\begin{aligned}
\frac{d}{dt}E_r(t)&=\int_{\mathbb T}\int_{B(r+t)}\partial_tu\partial_{tt}u+\sum_{i\in\{1,\dots,d\}}\partial_{x_i}u\,\partial_{x_i}\partial_tu+\partial_yu\,\partial_{y}\partial_tu\,dx\,dy\\
&\,\quad+\int_{\mathbb T}\int_{B(r+t)}u\partial_tu+\frac{1}{\a+2}|u|^{\a}u\partial_tu\,dx\,dy\\
&\,\quad+\frac12\int_{\mathbb T}\int_{\partial B(r+t)}\left(|\partial_tu|^2+|\nabla u|^2+|u|^2+\frac{2}{\a+2}|u|^{\a+2}\right)(t)\,d\sigma,\\
&=\int_{\mathbb T}\int_{B(r+t)}\partial_tu\partial_{tt}u+div_x(\partial_tu\nabla_xu)-\partial_tu\Delta_xu+\partial_y(\partial_tu\partial_yu)-\partial_tu\partial_{yy}u\\
&\,\quad+\int_{\mathbb T}\int_{B(r+t)}u\partial_tu+\frac{1}{\a+2}|u|^{\a}u\partial_tu\,dx\,dy\\
&\,\quad+\frac12\int_{\mathbb T}\int_{\partial B(r+t)}\left(|\partial_tu|^2+|\nabla u|^2+|u|^2+\frac{2}{\a+2}|u|^{\a+2}\right)(t)\,d\sigma,\\ 
&=\int_{\mathbb T}\,dy\int_{B(r+t)}div_x(\partial_tu\nabla_xu)+\int_{B(r+t)}\,dx\int_{\mathbb T}\partial_y(\partial_tu\partial_yu)\\
&\,\quad+\frac12\int_{\mathbb T}\int_{\partial B(r+t)}\left(|\partial_tu|^2+|\nabla u|^2+|u|^2+\frac{2}{\a+2}|u|^{\a+2}\right)(t)\,d\sigma,\\
&=-\int_{\mathbb T}\,dy\int_{\partial B(r+t)}\partial_tu\nabla u\cdot n_i\,d\sigma\\
&\,\quad+\frac12\int_{\mathbb T}\,dy\int_{\partial B(r+t)}\left(|\partial_tu|^2+|\nabla u|^2+|u|^2+\frac{2}{\a+2}|u|^{\a+2}\right)(t)\,d\sigma.
\end{aligned}
\end{equation*}
where $n_i=n_i(x),\,x\in\partial B,$ denotes the inner normal vector to the boundary of $B.$ Recall that the energy on the whole space in conserved, and so by using Cauchy-Schwartz inequality
\begin{equation*}
\begin{aligned}
\frac{d}{dt}(E-E_r(t))&=\frac{d}{dt}\left\{\frac12\int_{\mathbb T}\int_{B(r+t)^c}\left(|\partial_tu|^2+|\nabla u|^2+|u|^2+\frac{2}{\a+2}|u|^{\a+2}\right)(t)\,dx\,dy\right\},\\
&=\int_{\mathbb T}\int_{\partial B(r+t)}\partial_tu\nabla u\cdot n_i\,d\sigma\,dy\\
&\,\quad-\frac12\int_{\mathbb T}\int_{\partial B(r+t)}\left(|\partial_tu|^2+|\nabla u|^2+|u|^2+\frac{2}{\a+2}|u|^{\a+2}\right)(t)\,d\sigma\,dy,\\
&\leq\frac12\int_{\mathbb T}\int_{\partial B(r+t)}|\partial_tu|^2+|\nabla u|^2\,d\sigma\,dy\\
&\,\quad-\frac12\int_{\mathbb T}\int_{\partial B(r+t)}\left(|\partial_tu|^2+|\nabla u|^2+|u|^2+\frac{2}{\a+2}|u|^{\a+2}\right)(t)\,d\sigma\,dy\leq0,
\end{aligned}
\end{equation*}  
and we obtain  
\begin{equation*}
\frac{d}{dt}\left\{\frac12\int_{\mathbb T}\int_{B(r+t)^c}\left(|\partial_tu|^2+|\nabla u|^2+|u|^2+\frac{2}{\a+2}|u|^{\a+2}\right)(t)\,dx\,dy\right\}\leq0,
\end{equation*}
namely the energy on $B(x_0,r+t)^c\times\mathbb T$ is decreasing. The conclusion follows.
\end{proof} 
We now give an estimate from above of a portion away from zero of the potential energy. This will be essential in the last section dealing with the rigidity part in the Kenig \& Merle scheme. We follows the ideas of Bulut in \cite{Bul}, who in turn was inspired by Killip and Visan \cite{KV1, KV2, KV3}. We point out that in \cite{Bul} the situation is much more involved, since the author is considering energy supercritical NLW. 
\begin{lemma} 
Let $u(t,x,y)$ be a solution to \eqref{NLKG}. If $\{\vec{u}(t)\}_{t\in\R}\subset\mathcal H$ is a relatively compact set and $\vec{u}^*\in\mathcal H$ is one of its limit points, then $\vec{u}^*\neq0.$
\end{lemma}
\begin{proof}
This property simply follows from the conservation of energy \eqref{energy}.
\end{proof}
At this point we can give the following lemma, essentially based on the well-posedness of \eqref{NLKG}, in particular its continuous dependence on the initial data. 
\begin{lemma}\label{lem:6.4}
Let $u(t)$ be a nontrivial solution to \eqref{NLKG} such that $\{u(t,x-x(t),y),\partial_tu(t,x-x(t),y)\}_{t\in\R}$ is relatively compact in $\mathcal H.$ Then for any $A>0$, there exists $C(A)>0$ such that for any $t\in\R,$
\begin{equation}\label{eq:6.6}
\int_t^{t+A}\int_{\mathbb T}\int_{|x-x(t)|\leq R}|u|^{\a+2}(s,x,y)\,dx\,dy\,ds\geq C(A),
\end{equation}
for $R=R(A)$ large enough.
\end{lemma}
\begin{proof}
We argue by contradiction, supposing that there exists a sequence of times $\{t_n\}_{n\in\mathbb N}$ such that 
\begin{equation*}
\int_{t_n}^{t_n+A}\int_{\R^d\times\mathbb T}|u|^{\a+2}(s,x,y)\,dx\,dy\,ds<\frac1n.
\end{equation*}
By compactness, up to subsequence still denoted with the subscript $n$,  $$(u(t_n,x-x(t_n),y),\partial_tu(t_n,x-x(t_n),y))\ra(f,g)\in\mathcal H.$$ Let  $(w(0),\partial_tw(0))=(f,g)$ be an initial datum and $w(t)$ be the corresponding solution to \eqref{NLKG}: then we have, by the fact that $u\neq0,$
\begin{equation}\label{energylimit}
E(w,\partial_tw)=E(f,g)=\lim_{n\to\infty} E(u(t_n,x-x(t_n),y),\partial_tu(t_n,x-x(t_n),y))=E(u_0,u_1)\neq0.
\end{equation}
Local well-posedness and Strichartz estimates imply
\begin{equation*}
\begin{aligned}
0&=\lim_{n\to\infty}\int_{t_n}^{t_n+A}\int_{\R^d\times\mathbb T}|u|^{\a+2}(s,x,y)\,dx\,dy\,ds\\
&=\lim_{n\to\infty}\int_{0}^{A}\int_{\R^d\times\mathbb T}|u|^{\a+2}(t_n+s,x,y)\,dx\,dy\,ds\\
&=\lim_{n\to\infty}\int_{0}^{A}\int_{\R^d\times\mathbb T}|u|^{\a+2}(t_n+s,x-x(t_n),y)\,dx\,dy\,ds\\
&=\int_{0}^{A}\int_{\R^d\times\mathbb T}|w|^{\a+2}(s,x,y)\,dx\,dy\,ds,
\end{aligned}
\end{equation*}
which in turn gives that $w(t)=0$ almost everywhere in $(0,A).$ This contradicts \eqref{energylimit}, then $$\int_t^{t+A}\int_{\mathbb T\times\R^d}|u|^{\a+2}\,dxdydt\geq C^\prime(A).$$ By exploiting again the precompactness property of the solution
\begin{equation*}
\begin{aligned}
\int_t^{t+A}\int_{\mathbb T}\int_{|x-x(t)|\leq R}|u|^{\a+2}\,dxdydt&=\int_t^{t+A}\left\{\int_{\mathbb T\times\R^d}|u|^{\a+2}\,dxdy -\int_{\mathbb T}\int_{|x-x(t)|\geq R}|u|^{\a+2}\,dxdy\right\}\,dt,\\
&\geq C^\prime(A)-\frac{C^\prime(A)}{2}=\frac{C^\prime(A)}{2}=:C(A).
\end{aligned}
\end{equation*} 
\end{proof}

\begin{corollary}\label{l2}
By interpolation the same property can be claimed for the localized $L^2-$norm of $u.$ More precisely, under the same assumption of \autoref{lem:6.4} on $u,$ for any $A>0$ there exists $C(A)>0$ such that for any $t\in\R,$
\begin{equation}\label{eq:6.6bis}
\int_t^{t+A}\int_{\mathbb T}\int_{|x-x(t)|\leq R}|u|^2(s,x,y)\,dx\,dy\,ds\geq C(A)
\end{equation}
for $R=R(A)$ large enough.
\end{corollary}

The last ingredient to derive a contradiction to the existence of a such precompact solution is an a priori bound for the super-quadratic term of the energy which is due to Nakanishi, see \cite{Nak}. The latter is a remarkable extension in the euclidean framework $\mathbb R^{m}$ with $m=1,2$ of the well-known Morawetz estimate proved by Morawetz and Strauss, see \cite{Mor, MS}, for higher dimensions. This a priori bounds lead to the scattering in energy space both for the nonlinear Klein-Gordon equation and the nonlinear Schr\"odinger equation posed in the euclidean space.  
\subsection{Nakanishi/Morawetz-type estimate}\label{sec:morawetz}

We begin this section by giving the analogue in our domain of the decay result due to Nakanishi, \cite{Nak}. Our approach is to simply use a multiplier that does not consider all the variables: neither the compact factor of the product space we work on (the $y-$variable), nor a set of $d-1$ euclidean variables ($x_2, \; \dots \; ,x_d$ for instance) will be ``seen'' by the multiplier. 
Consequently, we will show how the Nakanishi/Morawetz type estimate in one dimension is enough for a contradiction argument which will exclude soliton-like solutions, i.e. the $u_c$ built in \autoref{exist:min}. \\

We report verbatim the proof contained in \cite{Nak}*{Lemma 5.1, equation (5.1)}, then we analyze the extra term given by the remaining part of the second order space operator involved in the equation. First, Nakanishi introduces the following term with relative notations (recall that in the following we are in a pure euclidean framework, with $x\in\R^m$ and $m=1,2$):
\begin{equation*}
r=|x|, \quad \theta=\frac{x}{r},\quad \lambda=\sqrt{t^2+r^2},\quad \Theta=\frac{(-t,x)}{\lambda}
\end{equation*}
\begin{equation*}
u_r=\theta\cdot\nabla_xu,\quad u_\theta=\nabla_xu-\theta u_r
\end{equation*} 
\begin{equation*}
l(u)=\frac12\left(-|\partial_t u|^2+|\nabla_xu|^2+|u|^2+\frac{2}{\a+2}|u|^{\a+2}\right)
\end{equation*} 
\begin{equation*}
(\partial_0,\partial_1,\partial_2)=(-\partial^0,\partial^1,\partial^2)=(\partial_t,\nabla_x)
\end{equation*} 
\begin{equation*}
g=\frac{m-1}{2\lambda}+\frac{t^2-r^2}{2\lambda^3},\quad M=\Theta\cdot(\partial_t u,\nabla_x u)+ug
\end{equation*} 
\begin{equation*}
(\partial_t^2-\Delta_x)g=-\frac{5}{2\lambda^3}+3\frac{t^2-r^2}{\lambda^5}+15\frac{(t^2-r^2)^2}{2\lambda^7}.
\end{equation*} 
\noindent Then by multiplying the equation $\partial_t^2u-\Delta_xu+u+|u|^{\a}u=0$ by $M,$ with $u=u(t,x),$ we obtain the relation
\begin{equation}\label{naka-ide}
\begin{aligned}
0=(\partial_t^2u-\Delta_xu+u+|u|^{\a}u)M&=\sum_{\beta=0}^m\partial_\beta\left(-M\partial^\beta u+l(u)\Theta_\beta +\frac{|u|^2}{2}\partial^\beta g\right)\\
&\,\quad+\frac{|u_\omega|^2}{\lambda}+\frac{|u|^2}{2}(\partial_t^2-\Delta_x)g+\frac{\a}{\a+2}|u|^{\a+2}g,
\end{aligned}
\end{equation}
where $u_\omega$ is the projection of $(\partial_tu,\nabla_xu)$ on the tangent space of $t^2-|x|^2=c,$ $c$ being a constant.\\

\noindent We focus on $m=1$ and we go back to \eqref{NLKG}. We introduce the compact notation $$\R^{d-1}\times\mathbb T=:\mathcal M\ni z:=(\bar x,y)=(x_2,\dots,x_d,y).$$ Then the analogous of \eqref{naka-ide} is the following:
\begin{equation}\label{naka-ide2}
\begin{aligned}
0=(\partial_t^2u-\Delta u+u+|u|^{\a}u)M&=\sum_{\beta\in\{0,1\}}\partial_\beta\left(-M\partial^\beta u+l(u)\Theta_\beta +\frac{|u|^2}{2}\partial^\beta g\right)\\
&\quad+\frac{|u_\omega|^2}{\lambda}+\frac{|u|^2}{2}(\partial_t^2-\Delta_x)g+\frac{\a}{\a+2}|u|^{\a+2}g\\
&\quad-M\Delta_zu.
\end{aligned}
\end{equation}
Observe  that the term $g$ is nonnegative only in the region where $r<t.$ Then after integrating \eqref{naka-ide2} (now $u=u(t,x_1,z)$) on $\mathcal C:=\{(t,x_1)\,|\,2<t<T, |x_1|=r<t\}\times\mathcal M,$ using the divergence theorem, the last relation we obtain is:
\begin{equation*}
\begin{aligned}
\left\{\int_{\mathcal M}\int_{r<t}-\partial_tuM+l(u)\frac{t}{\lambda}+\frac{|u|^2}{2}\partial_tg\,dx_1dz\right\}\bigg|^{t=T}_{t=2}=&\int_{\mathcal C}\frac{|u_\omega|^2}{\lambda}+\frac{|u|^2}{2}(\partial_t^2-\partial_{x_1}^2)g+\frac{\a}{\a+2}|u|^{\a+2}g\,dx_1dz\,dt\\
&\,+\frac{\sqrt 2}{2}\int_{\mathcal M}\int_{2<r=t<T}|u|^2+\frac{2}{\a+2}|u|^{\a+2}\,dx_1dz\\
&\,-\int_{\mathcal C}M\Delta_zu\,dx_1dz\,dt,
\end{aligned}
\end{equation*}
noticing that $|u_\theta|^2=0$ if $m=1.$ The l.h.s. of the above identity is bounded by the energy, as well as the middle term in the first integral in the r.h.s. thanks to the estimate
\begin{equation*}
\left|\int_{\mathcal C}\frac{|u|^2}{2}(\partial_t^2-\partial_{x_1}^2)g\right|\lesssim \int_2^T\int_{\mathbb T\times\R^d}\frac{|u|^2}{t^3}\,dx_1dzdt\lesssim E.
\end{equation*}
The energy flux through the curved surface, i.e. the second integral in the r.h.s. is estimated by the energy. In fact we have the following:
\begin{lemma}
Any smooth solution $u$ to \eqref{NLKG} satisfies:
\begin{equation}\label{flux}
\int_{\mathcal M}\int_{2<|x_1|=t<T}|\partial_tu-\theta\partial_{x_1}u|^2+|\nabla_zu|^2+|u|^2+\frac{2}{\a+2}|u|^{\a+2}\,d\sigma\,dz\lesssim E.
\end{equation}
\end{lemma}
\begin{proof}
The proof repeats the same analysis performed to prove the finite propagation speed property. Define
\begin{equation*}
e(t):=\frac12\int_{\mathcal M}\int_{|x_1|<t}\left(|\partial_tu|^2+|u|^2+|\partial_{x_1}u|^2+\frac{2}{\a+2}|u|^{\a+2}\right)(t,x_1,z)\,dx_1\,dz.
\end{equation*}  
Differentiating $e(t)$ with respect to $t$, we obtain
\begin{equation*}
\begin{aligned}
\frac{d}{dt}e(t)&=\int_{\mathcal M}\int_{|x_1|<t}\left(\partial_tu\partial_t^2u+\partial_{x_1}u\partial_{x_1}\partial_tu+\partial_tuu+|u|^\a u\partial_tu\right)\,dx_1\,dz\\
&\,\quad+\frac12\int_{\mathcal M}\int_{|x_1|=t}\left(|\partial_tu|^2+|u|^2+|\partial_{x_1}u|^2+\frac{2}{\a+2}|u|^{\a+2}\right)\,d\sigma\,dz,\\
&=\int_{\mathcal M}\int_{|x_1|<t}\partial_tu\left(\partial_t^2u-\partial_{x_1}^2u+u+|u|^\a u\right)+\partial_{x_1}(\partial_{x_1}u\cdot\partial_tu)\,dx_1\,dz\\
&\,\quad+\frac12\int_{\mathcal M}\int_{|x_1|=t}\left(|\partial_tu|^2+|u|^2+|\partial_{x_1}u|^2+\frac{2}{\a+2}|u|^{\a+2}\right)\,d\sigma\,dz,\\
&=\int_{\mathcal M}\int_{|x_1|<t}\partial_tu\left(\partial_t^2u-\Delta u+u+|u|^\a u\right)+\partial_tu\Delta_zu+\partial_{x_1}(\partial_{x_1}u\cdot\partial_tu)\,dx_1\,dz\\
&\,\quad+\frac12\int_{\mathcal M}\int_{|x_1|=t}\left(|\partial_tu|^2+|u|^2+|\partial_{x_1}u|^2+\frac{2}{\a+2}|u|^{\a+2}\right)\,d\sigma\,dz,\\
&=\int_{\mathcal M}\int_{|x_1|<t}\partial_tu\Delta_zu+\partial_{x_1}(\partial_{x_1}u\cdot\partial_tu)\,dx_1\,dz\\
&\,\quad+\frac12\int_{\mathcal M}\int_{|x_1|=t}\left(|\partial_tu|^2+|u|^2+|\partial_{x_1}u|^2+\frac{2}{\a+2}|u|^{\a+2}\right)\,d\sigma\,dz,\\
&=-\frac12\int_{\mathcal M}\int_{|x_1|<t}\partial_t|\nabla_zu|^2\,dx_1\,dz\\
&\,\quad+\frac12\int_{\mathcal M}\int_{|x_1|=t}\left(|\partial_tu|^2+|u|^2+|\partial_{x_1}u|^2+\frac{2}{\a+2}|u|^{\a+2}-2\theta\partial_{x_1}u\cdot\partial_tu\right)\,d\sigma\,dz,\\
&=-\frac12\frac{d}{dt}\int_{\mathcal M}\int_{|x_1|<t}|\nabla_zu|^2\,dx\,dz+ \frac12\int_{\mathcal M}\int_{|x_1|=t}|\nabla_zu|^2\,d\sigma\,dz\\
&\,\quad+\frac12\int_{\mathcal M}\int_{|x_1|=t}\left(|\partial_tu-\theta\partial_{x_1}u|^2+|u|^2+\frac{2}{\a+2}|u|^{\a+2}\right),
\end{aligned}
\end{equation*}
therefore, integrating with respect to the time variable from $2$ to $T$ we obtain \eqref{flux}.
\end{proof}
\noindent Moreover, the energy estimate on the surface of the light cone gives
\begin{equation*}
\sup_{t}\int_{\R^{d-1}\times\mathbb T}\int_{\R}|u(|x_1|+t,x_1,z)|^2\,dx_1dz\lesssim E.
\end{equation*}
We now analyze the term $-\int_{\mathcal C}M\Delta_zu\,dx_1dz\,dt$ in \eqref{naka-ide2}. We rewrite explicitly the term to be integrated as   
\begin{equation*}
-M\Delta_zu=-div_z\left(M\nabla_zu\right)+\nabla_zu\cdot\nabla_zM:=\mathcal A+\mathcal B.
\end{equation*}
The second term is explicitly given by 
\begin{equation*}
\begin{aligned}
\mathcal B&=-\frac{t}{2\lambda}\partial_t|\nabla_zu|^2+\frac{1}{2\lambda}x_1\cdot\partial_{x_1}|\nabla_zu|^2+g|\nabla_zu|^2,\\
&=-\frac12\partial_t\left(\frac{t}{\lambda}|\nabla_zu|^2\right)+\frac12|\nabla_zu|^2\partial_t\left(\frac{t}{\lambda}\right)+\frac{1}{2\lambda}\left(\partial_{x_1}(x_1|\nabla_zu|^2)-|\nabla_zu|^2\right)+g|\nabla_zu|^2,\\
&=-\frac12\partial_t\left(\frac{t}{\lambda}|\nabla_zu|^2\right)+\frac{|x_1|^2}{2\lambda^3}|\nabla_zu|^2+\frac{1}{2\lambda}\partial_{x_1}(x_1|\nabla_zu|^2)-\frac{|\nabla_zu|^2}{2\lambda}+g|\nabla_zu|^2,\\
&=-\frac12\partial_t\left(\frac{t}{\lambda}|\nabla_zu|^2\right)+\frac{|x_1|^2}{2\lambda^3}|\nabla_zu|^2+\partial_{x_1}\left(\frac{x_1}{2\lambda}|\nabla_zu|^2\right)-\partial_{x_1}\left(\frac{1}{2\lambda}\right) x_1|\nabla_zu|^2-\frac{|\nabla_zu|^2}{2\lambda}+g|\nabla_zu|^2,\\
&=-\frac12\partial_t\left(\frac{t}{\lambda}|\nabla_zu|^2\right)+\frac{|x_1|^2}{2\lambda^3}|\nabla_zu|^2+\partial_{x_1}\left(\frac{x_1}{2\lambda}|\nabla_zu|^2\right)+\frac{|x_1|^2}{2\lambda^3}|\nabla_zu|^2-\frac{|\nabla_zu|^2}{2\lambda}+g|\nabla_zu|^2,\\
&=-\frac12\partial_t\left(\frac{t}{\lambda}|\nabla_zu|^2\right)+\partial_{x_1}\left(\frac{x_1}{2\lambda}|\nabla_zu|^2\right),
\end{aligned}
\end{equation*}
\noindent and then, after integration, it can be estimated by the energy on the whole space, while the divergence term $\mathcal B$ disappears using the Gauss-Green theorem.
In conclusion  
\begin{equation*}
\int_2^\infty\int_{\mathcal M\times\{|x_1|<t\}}\frac{|u_\omega|^2}{\lambda}+\frac{\a}{\a+2}|u|^{\a+2}g\,dx_1dzdt\lesssim E.
\end{equation*}

\noindent The Nakanishi/Morawetz-type estimate follows as in \cite{Nak}:
\begin{equation}\label{na-mo}
\int_{\R}\int_{\R^{d}\times\mathbb T}\frac{\min\{|u|^2, |u|^{\a+2}\}}{\langle t\rangle\log(|t|+2)\log(\max\{|x_1|-t,2\})}\lesssim E.
\end{equation}


\noindent We have now all the elements allowing the exclusion of the soliton-like solution. 
\subsection{Extinction of the minimal element} With the aforementioned tool, we are in position to obtain a contradiction with respect to the our hypothesis on the finiteness of the critical energy $E_c.$ Consider the upper bound $C=C(E(u))$ appearing in \eqref{na-mo}, then for any $T>2$ we can write
\begin{equation}\label{contra-estimate}
\begin{aligned}
C&\geq \int_{\R}\int_{\R^{d-1}\times\mathbb T}\int_\R\frac{\min\{|u|^2, |u|^{\a+2}\}}{\langle t\rangle\log(|t|+2)\log(\max\{|x_1|-t,2\})}\,dx_1dzdt,\\
&\geq\int_2^T\int_{\R^{d-1}\times\mathbb T}\int_\R\frac{\min\{|u|^2, |u|^{\a+2}\}}{\langle t\rangle\log(|t|+2)\log(\max\{|x_1|-t,2\})}\,dx_1dzdt,\\
&\geq\int_2^T\int_{\mathbb T}\int_{|x-x(t)|\leq R}\frac{\min\{|u|^2, |u|^{\a+2}\}}{\langle t\rangle\log(|t|+2)\log(\max\{|x_1|-t,2\})}\,dx_1dzdt.
\end{aligned}
\end{equation}
The finite propagation speed implies that $|x(t)-x(0)|\leq t+c_0$ for $t>0,$ then 
\begin{equation*}
|x|\leq |x-x(t)|+|x(t)-x(0)|+|x(0)|\leq R+t+c_0+c_1,
\end{equation*}
so that $|x_1|-t\leq c+R.$ 
With $[T]$ being the usual floor function of $T,$ we are able to carry on with the chain above with 
\begin{equation*}
\begin{aligned}
\eqref{contra-estimate}&\gtrsim\int_2^T\frac{1}{\langle t\rangle\log(|t|+2)}\int_{\mathbb T}\int_{|x-x(t)|\leq R}\min\{|u|^2, |u|^{\a+2}\}\,dxdt,\\
&\gtrsim\int_2^{[T]}\frac{1}{\langle t\rangle\log(|t|+2)}\int_{\mathbb T}\int_{|x-x(t)|\leq R}\min\{|u|^2, |u|^{\a+2}\}\,dxdt,\\
&=\sum_{j=3}^{[T]}\int_{j-1}^{j}\frac{1}{\langle t\rangle\log(|t|+2)}\int_{\mathbb T}\int_{|x-x(t)|\leq R}\min\{|u|^2, |u|^{\a+2}\}\,dxdt,\\
&\gtrsim\sum_{j=3}^{[T]}\frac{1}{\langle j\rangle\log (j+2)}\int_{j-1}^{j}\int_{\mathbb T}\int_{|x-x(t)|\leq R}\min\{|u|^2,|u|^{\a+2}\}\,dxdt,\\
&\gtrsim C(1)\sum_{j=3}^{[T]}\frac{1}{\langle j\rangle\log (j+2)}\sim\int_{2}^{T}\frac{1}{\langle t\rangle\log(t+2)}\,dt.
\end{aligned}
\end{equation*}

\noindent In the last step we used the property stated in \autoref{lem:6.4} and \autoref{l2} (more precisely \eqref{eq:6.6} and \eqref{eq:6.6bis}) above for a suitable choice of the radius $R.$ This is sufficient establish a contradiction by taking $T$ large enough, since for $T\ra+\infty$
\begin{equation*}
\int_{2}^{T}\frac{1}{\langle t\rangle\log t}\,dt\sim\int_{2}^{\infty}\frac{1}{t\log t}\,dt,
\end{equation*}
and the latter diverges, while the chain of inequalities above should imply a uniform bound. 

\section*{Acknowledgments}
\noindent The authors wish to express their gratitude to Professor Nicola Visciglia for suggesting the problem and for
his support during the preparation of this paper. The authors would also like to thank Professor Kenji
Nakanishi for his kind and prompt availability to explain some points of \cite{IMN11}, and for his stimulating remarks about our paper. 
L.H. was partially supported by  ``INDAM-COFUND: Project NLDISC'' and ``Chrysalide: Comportements asymptotiques de solutions d'EDP non-lin\'eaires dispersives'' from UBFC.
\appendix

\section{Interpolation}

\begin{lemma}\label{app:lemma1}
Let $\alpha \in \left(\frac4d,\frac{4}{d-1}\right)$ for $2\leq d\leq4$ or $\a>4$ if $d=1.$ Consider $\{u_n\}_{n\in\mathbb N}$ a sequence of solutions to 
\begin{equation*}\left\{\begin{aligned}
\partial_{tt}u_n-\Delta_{x,y}u_n+u_n&=0,\qquad (t,x,y) \in \R\times \R^d\times \mathbb T\\
u_n(0,x,y)&=f_n(x,y)\in H^1(\R^d\times \mathbb T)\\
\partial_tu_n(0,x,y)&=g_n(x,y)\in L^2(\R^d\times \mathbb T)
\end{aligned}\right. ,
\end{equation*} 
	 with $\sup_n\|\left(f_n,g_n\right)\|_{\mathcal H}\leq C<\infty.$ Suppose that for any $q\in (2,2^*),$ with $2^*$ defined in \eqref{sob}
\begin{equation*}
 \lim_{n\to\infty}\|u_n\|_{L^\infty L^q}= 0.
 \end{equation*}
		Then \begin{equation*}
		\lim_{n\to\infty}\|u_n\|_{L^{\alpha+1}L^{2(\alpha+1)}}=0.
		\end{equation*}
\end{lemma}
\begin{proof}
	We drop the subscript $n$ to lighten the notations. We first make a formal computation (from H\"older inequality) without adjusting the parameters:
	\begin{align}
	\nonumber\|u\|_{L^{\alpha+1}L^{2(\alpha+1)}}& = \left(\int \left(\int |u|^a |u|^b \right)^{1/2}\right)^{1/(\alpha+1)}\\
	\nonumber							& \leq \left(\int \left(\int |u|^{ar}  \right)^{1/(2r)}\; \left(\int |u|^{bs}  \right)^{1/(2s)} \right)^{1/(\alpha+1)}\\
	\nonumber							& \leq \left( \|u\|_{L^{ar}}^{a/2} \; \|u\|_{L^{bs}}^{b/2}\right)^{1/(\alpha+1)}\\
	\label{holder}						& \leq  \|u\|_{L^{\infty}L^{ar}}^{a/(2\alpha+2)}  \|u\|_{L^{b/2}L^{bs}}^{b/(2\alpha+2)}.
	\end{align}
	The claim of \autoref{app:lemma1} is satisfied if the following conditions are fulfilled in \eqref{holder}:
	\begin{align}
	\label{c1}		&a+b = 2\alpha+2, a>0, b>0, \\
	\label{c2}		&r=q/a >1, \\
	\label{c3}		&s= r/(r-1),\\
	\label{c4}		&(b/2, bs)\quad \text{is a Strichartz pair as in}\, \autoref{strichartzRd}.
	\end{align}
	Under these conditions, we may have by hypothesis along with the energy conservation
	$$\|u_n\|_{L^{\alpha+1}L^{2\alpha+2}} \leq \|u_n\|^\gamma_{L^{\infty}L^{q}} E^{1-\gamma}\rightarrow 0, $$
	where $\gamma\in(0,1).$ Note that it is enough to have the convergence to zero in only one $L^\infty L^q$.
	\\
	
	\noindent Let us now check that all conditions are non-empty, by separating the two cases $d=1$ and $2\leq d\leq4.$
	\\
	
\noindent \emph{Case $d=1.$} On one hand, conditions \eqref{c1},\eqref{c2} give
	\begin{equation*}
	2\alpha+2-q<b<2\alpha+2.
	\end{equation*}
To satisfy the Strichartz admissibility \eqref{c4}, we impose to have $b/2>4$ and $bs\geq \frac{2b}{b-8}$ and therefore
	\begin{equation}\label{a:6}
	b>8,\qquad \frac{q}{q-2\a-2+b}\geq\frac{2}{b-8}\iff q\geq\frac{2b-4\a-4}{b-10}.
	\end{equation}
	Thus, how to concern the choice of $b,$ it is enough to have 
	\begin{equation}
	\label{c5}
	\max(8,2\alpha+2-q)<b<2\a+2.
	\end{equation}
	
\noindent On the other hand, the second condition in \eqref{a:6} allows us to strengthen \eqref{c5} to be 
\begin{equation}
	\label{c51}
	\max(10,2\alpha+2-q)<b<2\a+2.
	\end{equation}
Since $\frac{2b-4\a-4}{b-10}<2$ for any $\a>4$ and for any $\max(10,2\alpha+2-q)<b<2\a+2,$  we claim that for any $q>2$ we can find $b$ such that condition \eqref{a:6} is satisfied. \\

\noindent\emph{Case $2\leq d\leq4.$} In this situation we must consider that $\a\in\left(\frac{4}{d},\frac{4}{d-1}\right)$ and then $$2\a+2\in\left(\frac{8+2d}{d},\frac{6+2d}{d-1}\right).$$

\noindent The Strichartz admissibility now reads $b>4$ if $d=2$ or $b\geq4$ if $d=3,4$ that we strengthen as $b>4,$ and  
\begin{equation*}
\frac{2db}{db-8}\leq bs \leq\frac{2b(d+1)}{b(d-1)-4},
\end{equation*} which can be rewritten in term of $q$ as 
\begin{equation}\label{q-range}
b>4,\qquad \frac{2d}{db-8}\leq \frac{q}{q-2\a-2+b} \leq\frac{2(d+1)}{b(d-1)-4}.
\end{equation}

\noindent Let us focus on the l.h.s. inequality  \begin{equation}\label{q-pos}
\frac{2d}{db-8}\leq \frac{q}{q-2\a-2+b}\iff q\geq\frac{2d(2\a+2-b)}{2d+8-bd}.
\end{equation} 
By strengthening the condition on $b$ in the following way
\begin{equation*}
\max\left\{\frac{8+2d}{d},2\a+2-q\right\}<b<2\a+2,
\end{equation*}
we obtain that $\frac{2d(2\a+2-b)}{2d+8-bd}<0$ and since $q$ is positive clearly satisfies  \eqref{q-pos}. We turn the attention to the r.h.s. inequality of \eqref{q-range}. It is equivalent to 
\begin{equation}\label{q-right}
q\geq \frac{2(d+1)(2\a+2-b)}{6+2d-b(d-1)}.
\end{equation}
If we had $\frac{2(d+1)(2\a+2-b)}{6+2d-b(d-1)}<2$ we could conclude. Computations show that 
\begin{equation*}
\frac{2(d+1)(2\a+2-b)}{6+2d-b(d-1)}<2\iff b>\a(d+1)-2.
\end{equation*}
This is straightforward if we had $4>\a(d+1)-2$ which is in turn implied by the condition $\frac{6}{d+1}<\frac{4}{d-1}$ which is equivalent to $d<5.$ 
The proof is concluded. 
\end{proof}

\section{Decay property}\label{app-decay}
We use a decay property from \cite{D'ancona}. The key argument is, again, a scaling argument as in the \autoref{sec:Strichartz} and \cite{HV}. We will briefly sketch the proof given in  \cite[Example 1.2]{D'ancona}.
\\

By means of the basis $\{\Phi_j(y)\}_{j\in\mathbb N}$ given in \eqref{basis} and  \eqref{decompo} we decompose 
\begin{equation*}
e^{it\sqrt{1-\Delta_{x,y}}} f(x,y) = \sum_{j\in \mathbb N} e^{it \sqrt{1+\lambda_j-\Delta_x}}f_j(x)\Phi_j(y).
\end{equation*}
Thus  we get
\begin{equation*}
\left\|e^{it\sqrt{1-\Delta_{x,y}}}f \right\|_{L^\infty(\R^d\times \mathbb{T})}\leq \sum_{j\in \mathbb N} \left\|e^{it \sqrt{1+\lambda_j-\Delta_x}}f_j(\cdot)\right\|_{L^\infty(\R^d)}\|\Phi_j(\cdot)\|_{L^\infty_y}.
\end{equation*}
From \cite{D'ancona}, we have
\begin{equation*}
\left\|e^{it\sqrt{1-\Delta_x}}f\right\|_{L^\infty_x} \leq C |t|^{-d/2} \|f\|_{B^{\frac{d}{2}+1}_{1,1}}.
\end{equation*}
The function $w_m(t,x)=e^{it \sqrt{m-\Delta_x}}f$ satisfies the equation $\partial_{tt}w_m-\Delta_xw_m+mw_m=0$ with $w(0,x)=f(\sqrt m x):=f_m,$ with $w_m:=w(\sqrt mt,\sqrt mx)$ and $w$ satisfying $\partial_{tt}w-\Delta_xw+w=0$ with $w(0,x)=f(x).$ We use a scaling argument to deduce an estimate for $f_m$, noticing that for $m \geq 1$, 
the Besov norm of a rescaled function can be bounded by:
\begin{equation*}
\|f_m\|_{B^{\frac{d}{2}+1}_{1,1}}\leq m^{\frac{d+2}{4}}\|f\|_{B^{\frac{d}{2}+1}_{1,1}},
\end{equation*}
giving the following estimate with $m=1+\lambda_j>1$
\begin{equation*}
\begin{aligned}
\left\|e^{it\sqrt{1-\Delta_{x,y}}} f\right\|_{L^\infty(\R^d\times \mathbb{T})}&\leq C|t|^{-\frac{d}{2}}\sum_{j\in\mathbb N} 
\sqrt{1+\lambda_j}\|f_j\|_{B^{\frac{d}{2}+1}_{1,1}}\|\Phi_j(y)\|_{L^\infty_y}\\
&= C|t|^{-\frac{d}{2}}\sum_{j\in\mathbb N} 
(1+\lambda_j)^{d+1}(1+\lambda_j)^{-d-1/2}\|f_j\|_{B^{\frac{d}{2}+1}_{1,1}}\|\Phi_j(y)\|_{L^\infty_y}\\
&\lesssim |t|^{-\frac{d}{2}}\sum_{j\in\mathbb N} 
(1+\lambda_j)^{d+1}\|f_j\|_{B^{\frac{d}{2}+1}_{1,1}}\|\Phi_j(y)\|_{L^\infty_y}.
\end{aligned}
\end{equation*}
Noticing that the right handside can be expressed a term involving derivatives in $(x,y)$, one can find $N\in\mathbb N$ large enough to have
\begin{equation}
\label{disp:m3}
\left\|e^{it\sqrt{1-\Delta_{x,y}}}f \right\|_{L^\infty(\R^d\times \mathbb{T})}\leq C|t|^{-\frac{d}{2}} \|f\|_{W^{N,1}(\R^d\times \mathbb{T})}.
\end{equation}




\end{document}